\newif\ifsiam
\siamfalse

\ifsiam
   \documentclass[review]{siamart190516}
   \usepackage[numbers]{natbib}
\else
   \documentclass{amsart}
   \usepackage[margin=1in]{geometry}
   \usepackage{amsthm}
   \usepackage[foot]{amsaddr}
\fi

\usepackage{amssymb}
\usepackage{bm}
\usepackage{booktabs}
\usepackage{enumitem}
\usepackage{graphicx}
\usepackage{hyperref}
\usepackage{cleveref}
\usepackage{tikz}
\usepackage{xspace}
\usepackage{mathrsfs}
\usepackage{dutchcal}
\usepackage{multirow}
\usepackage{siunitx}
\usepackage{algorithm}
\usepackage{algpseudocode}
\usepackage[normalem]{ulem} 

\usetikzlibrary{trees}

\ifsiam
   \newsiamremark{remark}{Remark}
   \def\myqed{}
\else
   \usepackage[giveninits,date=year,sortcites]{biblatex}
   \DeclareCiteCommand{\abstractcite}%
   {\usebibmacro{prenote}}%
   {\printfield[journaltitle]{shortjournal},
    \printfield{volume}
    (\printfield{year}),
    \printfield{pages}}
   {\multicitedelim}
   {\usebibmacro{postnote}}

   \newtheorem{proposition}{Proposition}
   \theoremstyle{remark}
   \newtheorem{remark}{Remark}
   \def\myqed{\qedhere}
\fi

\let\div\undefined
\DeclareMathOperator{\diag}{diag}
\DeclareMathOperator{\div}{div}

\def\Hdiv{\bm{H}(\div)}

\def\T{\mathcal{T}}
\def\k{\kappa}

\def\tr{\mathsf{T}}

\def\Vh{\bm V_h}

\def\A{\mathsf{A}}
\def\B{\mathsf{B}}
\def\C{\mathsf{C}}
\def\D{\mathsf{D}}
\def\M{\mathsf{M}}
\def\W{\mathsf{W}}
\def\S{\mathsf{S}}
\def\I{\mathsf{I}}
\def\Zero{\mathsf{0}}

\def\L{\mathsf{L}}

\def\u{\boldsymbol{\mathsf{u}}}
\def\v{\boldsymbol{\mathsf{v}}}
\def\f{\boldsymbol{\mathsf{f}}}
\def\g{\mathsf{g}}
\def\p{\mathsf{p}}
\def\q{\mathsf{q}}
\def\r{\mathsf{r}}

\def\w{\mathsf{w}}

\def\x{\mathsf{x}}
\def\y{\mathsf{y}}

\def\AA{\mathcal{A}}
\def\BB{\mathcal{B}}
\def\xx{\mathcal{x}}

\usepackage{mdframed}
\newmdenv[fontcolor=red,linecolor=red,skipabove=12pt,skipbelow=12pt,
        innertopmargin=12pt,innerbottommargin=12pt]{note}

\newcommand{\hdr}[1]{\multicolumn{1}{c}{#1}}

\ifsiam
\title{Matrix-free GPU-accelerated saddle-point solvers for high-order problems in $\Hdiv$}
\author{Will Pazner\thanks{Fariborz Maseeh Department of Mathematics and Statistics, Portland State University, Portland, OR (\email{pazner@pdx.edu}, \email{panayot@pdx.edu})}
\and Tzanio Kolev\thanks{Center for Applied Scientific Computing, Lawrence Livermore National Laboratory, Livermore, CA (\email{kolev1@llnl.gov})}
\and Panayot Vassilevski\footnotemark[1]}
\else
\title[Matrix-free $\Hdiv$ saddle-point solvers]{Matrix-free GPU-accelerated saddle-point solvers for high-order problems in $\Hdiv$}

\author[Pazner]{Will Pazner$^1$}
\author[Kolev]{Tzanio Kolev$^2$}
\author[Vassilevski]{Panayot Vassilevski$^1$}
\address{$^1$Fariborz Maseeh Department of Mathematics and Statistics, Portland State University, Portland, OR}
\address{$^2$Center for Applied Scientific Computing, Lawrence Livermore National Laboratory, Livermore, CA}
\fi

\def\theabstract{
   This work describes the development of matrix-free GPU-accelerated solvers for high-order finite element problems in $\Hdiv$.
   The solvers are applicable to grad-div and Darcy problems in saddle-point formulation, and have applications in radiation diffusion and porous media flow problems, among others.
   Using the interpolation--histopolation basis (cf.~\abstractcite{Pazner2023}), efficient matrix-free preconditioners can be constructed for the $(1,1)$-block and Schur complement of the block system.
   With these approximations, block-preconditioned MINRES converges in a number of iterations that is independent of the mesh size and polynomial degree.
   The approximate Schur complement takes the form of an M-matrix graph Laplacian, and therefore can be well-preconditioned by highly scalable algebraic multigrid methods.
   High-performance GPU-accelerated algorithms for all components of the solution algorithm are developed, discussed, and benchmarked.
   Numerical results are presented on a number of challenging test cases, including the ``crooked pipe'' grad-div problem, the SPE10 reservoir modeling benchmark problem, and a nonlinear radiation diffusion test case.
}

\begin{document}

\ifsiam
   \maketitle
   \begin{abstract}
      \theabstract
   \end{abstract}
\else
   \begin{abstract}
      \theabstract
   \end{abstract}
   \maketitle
\fi

\section{Introduction}
\label{sec:intro}

This paper is concerned with the efficient solution of high-order finite element problems posed in the space $\Hdiv$;
in particular, we are interested in the grad--div problem
\begin{equation}
   \label{eq:grad-div}
   -\nabla (\alpha \nabla \cdot \bm u) + \beta \bm u = \bm f,
\end{equation}
and the closely related Darcy-type problem
\begin{equation}
   \label{eq:darcy}
   \left\{
   \begin{aligned}
      \bm u + \varepsilon \nabla p &= \bm f, \\
      \nabla \cdot \bm u + \gamma p &= g,
   \end{aligned}
   \right.
\end{equation}
where $\bm u \colon \Omega \to \mathbb{R}^d$ is a vector-valued unknown, $p \colon \Omega \to \mathbb{R}$ is a scalar unknown ($\Omega \subseteq \mathbb{R}^d$, $d \in \{2, 3\}$), and $\alpha, \beta, \gamma$ and $\varepsilon$ are potentially heterogeneous scalar coefficients.
Such problems arise in numerous applications, including, for example, radiation diffusion and porous media flow, among others \cite{Brunner2013,Boffi2013}.
We consider finite element discretizations for \eqref{eq:grad-div} and \eqref{eq:darcy}, where $\bm u$ is approximated in the Raviart--Thomas $\Hdiv$ finite element space, and $p$ is approximated in the discontinuous $L^2$ finite element space.

The solution of such problems via high-order finite element methods presents several advantageous features, as well as some challenges.
High-order methods promise high accuracy per degree of freedom, and have been demonstrated to achieve high performance on modern supercomputing architectures, in particular those based on graphics processing units (GPUs) and other accelerators \cite{Abdelfattah2021,Kolev2021a,LORGPU2022}.
However, the solution of the resulting large linear systems remains challenging for a number of reasons.
With increasing polynomial degree, the systems become both more ill-conditioned, and less sparse: the number of nonzeros in the matrix scales like $\mathcal{O}(p^{2d})$ and the condition number typically scales like $\mathcal{O}(p^4/h^2)$, where $p$ is the polynomial degree and $d$ is the spatial dimension \cite{Melenk2002}.
Consequently, it is necessary to develop effective matrix-free preconditioners that give rise to well-conditioned systems without having access to an assembled matrix.
There are a number of existing approaches for matrix-free preconditioning elliptic equations;
among these are $p$-multigrid, low-order-refined (also known as FEM--SEM) preconditioning, and hybrid Schwarz preconditioning \cite{Kronbichler2019,Pazner2020a,Lottes2005}.
However, the extension of these approaches to problems in $\Hdiv$ is not straightforward;
for example, multigrid-type methods with simple point smoothers are known not to perform well for such problems \cite{Arnold1997}.

Effective preconditioners for low-order $\Hdiv$ discretizations include multigrid with Schwarz smoothers \cite{Arnold2000}, auxiliary space algebraic multigrid preconditioning (e.g.\ the auxiliary divergence solver, ADS) \cite{Kolev2012}, and algebraic hybridization \cite{Dobrev2019}.
We briefly describe some existing approaches for extending these approaches to high-order discretizations.
In \cite{Pazner2023}, a low-order-refined approach for preconditioning grad--div problems in $\Hdiv$ was proposed.
In that work, the high-order operator discretized using the so-called interpolation--histopolation basis was shown to be spectrally equivalent to the lowest-order discretization posed on a refined mesh.
Then, an effective preconditioner for the low-order discretization (for example, the AMG-based ADS solver) can be used to precondition the high-order operator.
In \cite{Brubeck2022}, a vertex-star domain decomposition was combined with the use of a finite element basis with favorable sparsity properties to develop optimal-complexity multigrid solvers.

In the present work, we develop efficient matrix-free preconditioners for high-order $\Hdiv$ problems by considering the associated $2\times2$ saddle-point system.
Such saddle-point systems can be preconditioned effectively using block preconditioners \cite{Benzi2005,Zulehner2011,Pestana2015}.
The main challenge associated with developing such methods is the construction of spectrally equivalent approximations to the inverse of the Schur complement.
Using favorable properties of the interpolation--histopolation basis (cf.~\cite{Pazner2023}) it is possible to both enhance the sparsity of the block system, and to form sparse M-matrix approximations to the Schur complement that are effectively preconditioned with standard algebraic multigrid methods.
Although the block system is indefinite and contains more unknowns that the positive-definite system obtained by discretizing the grad--div problem directly, it is much less computationally expensive to precondition, resulting in significantly decreased cost per iteration, and faster overall solve time compared with alternative approaches.
Additionally, the described methods have low memory requirements, are highly scalable, and are amenable to GPU acceleration.

The structure of this paper is as follows.
In \Cref{sec:discretization} we describe the high-order finite element discretization and the associated linear systems.
This section also describes the construction of the block preconditioners, and includes results concerning the estimation of the condition number of the preconditioned system.
\Cref{sec:alg-gpu} describes the algorithmic details and GPU acceleration of the proposed methods.
Numerical examples are presented in \Cref{sec:results};
these include several challenging grad--div and Darcy benchmark problems, as well as a fully nonlinear time-dependent radiation-diffusion example.
We end with conclusions in \Cref{sec:conclusions}.

\section{Discretization and linear system}
\label{sec:discretization}

In the discrete setting, we consider a mesh $\T$ of the domain $\Omega$, where each element $\k \in \T$ is the image of the unit cube $[0,1]^d$ under a (typically isoparametric) transformation $T_\k$ (i.e., $\k = T_{\k}([0,1]^d)$).
In other words, we consider quadrilateral meshes in two spatial dimensions, and hexahedral meshes in three dimensions.
The variational problem associated with \eqref{eq:grad-div} is given by
\begin{equation}
   \label{eq:variational-problem}
   (\alpha \nabla \cdot \bm u, \nabla \cdot \bm v) + (\beta \bm u, \bm v) = (\bm f, \bm v).
\end{equation}
To discretize \eqref{eq:variational-problem}, we take $\bm u, \bm v \in \Vh$, where $\Vh$ is the degree-$p$ Raviart--Thomas finite element space,
\[
   \Vh = \{\, \bm{x} \in \Hdiv : \bm v|_\k \circ T_\k \in \Vh(\k) \,\}.
\]
The local space $\Vh(\k)$ is the image of the reference space $\Vh(\widehat{\k})$ under the $\Hdiv$ Piola transformation $\Vh(\k) = \det(J_\k)^{-1} J_\k \Vh(\widehat{\k})$, where $J_\k$ is the Jacobian of the element transformation $T_\k$.
In three dimensions, the reference space is given by $\Vh(\widehat{\k}) = \mathcal{Q}_{p,p-1,p-1} \times \mathcal{Q}_{p-1,p,p-1} \times \mathcal{Q}_{p-1,p-1,p}$, where $\mathcal{Q}_{p_1,p_2,p_3}$ is the space of trivariate polynomials of degree at most $p_i$ in variable $x_i$.
$\Vh(\widehat{\k})$ is defined analogously in two dimensions.

After discretization, the linear system of equations
\begin{equation}
   \label{eq:Au-eq-f}
   \A \u = \f
\end{equation}
is obtained, where $\A$ is the matrix given by the sum of the $\Hdiv$ stiffness and mass matrices, $\u$ is a vector of degrees of freedom representing $\bm u$, and $\f$ is a dual vector representing $\bm f$.

The grad--div problem \eqref{eq:grad-div} can be equivalently recast as the first-order system
\begin{equation}
   \label{eq:first-order-system}
   \left\{
   \begin{aligned}
      \beta \bm u - \nabla (\alpha q) &= \bm f, \\
      \nabla \cdot \bm u - q &= 0,
   \end{aligned}
   \right.
\end{equation}
leading to the symmetric variational problem
\begin{equation}
   \label{eq:block-variational}
   \begin{aligned}
      (\beta \bm u, \bm \sigma) + (\alpha q, \nabla \cdot \bm \sigma) &= (\bm f, \bm \sigma), \\
      (\alpha \nabla \cdot \bm u, \tau) - (\alpha q, \tau) &= 0,
   \end{aligned}
\end{equation}
where $\bm u, \bm \sigma \in \Vh$ and $q, \tau \in W_h$, where $W_h$ is the degree-$(p-1)$ $L^2$ finite element space
\[
   W_h = \{\, w \in L^2(\Omega) : w|_\k \circ T_\k \in \W_h(\k)  \,\},
\]
The local space $W_h(\kappa)$ is the image of the reference space $W_h(\widehat{\kappa}) = \mathcal{Q}_{p-1}$ under the integral-preserving mapping, $W_h(\kappa) = \det(J_{\kappa})^{-1} W_h(\widehat{\kappa})$.
Note that the space $W_h$ is equal to the range of divergence operator acting on $\Vh$

The variational system \eqref{eq:block-variational} can be written in matrix form as
\begin{equation}
   \label{eq:saddle-point}
   \begin{pmatrix}
      \M_\beta & \B_\alpha^\tr \\
      \B_\alpha & -\W_\alpha
   \end{pmatrix}
   \begin{pmatrix} \u \\ \q \end{pmatrix}
   =
   \begin{pmatrix} \f \\ \Zero \end{pmatrix},
\end{equation}
where the matrices $\M$, $\B$, and $\W$ are defined by
\begin{equation}
   \label{eq:matrices}
   \begin{aligned}
      \v^\tr \M_\beta \u &= (\beta \bm u, \bm v), \\
      \q^\tr \B_\alpha \u &= (\alpha \nabla \cdot \bm u, q), \\
      \r^\tr \W_\alpha \q &= (\alpha q, r).
   \end{aligned}
\end{equation}
It is straightforward to see that eliminating the scalar unknown $q$ from \eqref{eq:first-order-system} results in the original grad--div system \eqref{eq:grad-div}.
Likewise, the Schur complement of the saddle-point matrix with respect to the $(2,2)$-block $\M_\beta + \B_\alpha^\tr \W_\alpha^{-1} \B_\alpha$ is equal to the original $\Hdiv$ matrix $\A$.

Note that the block system \eqref{eq:saddle-point} is also closely related to the discretized system corresponding to the Darcy problem \eqref{eq:darcy}.
In the case of the mixed finite element problem, the resulting system takes the form
\begin{equation}
   \label{eq:darcy-nonzero}
   \begin{pmatrix}
      \M_{1/\varepsilon}  & \B^\tr \\
      \B  & -\W_\gamma
   \end{pmatrix}
   \begin{pmatrix} \u \\ \p \end{pmatrix}
   =
   \begin{pmatrix} \f \\ \g \end{pmatrix},
\end{equation}
where $\M_{1/\varepsilon}$ is the $\Hdiv$ mass matrix weighted by $\varepsilon^{-1}$, $\B$ corresponds to the unweighted divergence form, and $\W_\gamma$ is the $L^2$ mass matrix weighted by $\gamma$ (cf.~\eqref{eq:matrices}).
In contrast to the variational form corresponding to the grad-div problem, in this context the coefficient $\gamma$ may be zero (corresponding to the Poisson problem with no reaction term), leading to the saddle-point system with zero $(2,2)$-block
\begin{equation}
   \label{eq:darcy-zero}
   \begin{pmatrix}
      \M_{1/\varepsilon}  & \B^\tr \\
      \B  & \Zero
   \end{pmatrix}
   \begin{pmatrix} \u \\ \p \end{pmatrix}
   =
   \begin{pmatrix} \f \\ \g \end{pmatrix}.
\end{equation}

In what follows, we focus on the solution of the $2 \times 2$ saddle-point systems \eqref{eq:saddle-point} and \eqref{eq:darcy-nonzero}.
While the original grad-div system is symmetric and positive-definite, and so may be solved using the conjugate gradient method, the saddle-point system is symmetric and indefinite, and so we use the MINRES Krylov method (or GMRES, when using a non-symmetric preconditioner).
To precondition this system, we make use of block preconditioners developed in \Cref{sec:block-precond}.
The preconditioners for the grad-div saddle-point system \eqref{eq:saddle-point} apply equally well to the Darcy block systems \eqref{eq:darcy-nonzero} and \eqref{eq:darcy-zero}, including the case of zero $(2,2)$-block.

\subsection{Interpolation-histopolation basis}
\label{sec:basis}

In order to construct efficient preconditioners for the block systems, we will use properties of a particular choice of bases for the high-order Raviart--Thomas space $\Vh$ and $L^2$ space $W_h$.
These spaces will use the ``interpolation-histopolation'' basis (see \cite{Pazner2023} and also \cite{Dohrmann2021a,Kreeft2011,LORGPU2022}).
To define this basis, the $d$-dimensional reference element $[0,1]^d$ is subdivided into $p^d$ subelements (where $p$ is the polynomial degree) whose vertices are given by the $d$-fold Cartesian product of the $p+1$ Gauss--Lobatto points (see \Cref{fig:lor-cube}).
Then, the basis is induced by the following choice of degrees of freedom:
\begin{itemize}
   \item The degrees of freedom of the Raviart--Thomas space $\Vh$ are given by integrals of the normal component over each face of the subelement mesh. The basis functions are induced by the tensor-product of one-dimensional interpolation and histopolation operators.
   \item The degrees of freedom of the $L^2$ space $W_h$ are given by integrals over each subelement volume. The basis functions are induced by the tensor-product of one-dimensional histopolation operators.
\end{itemize}
\begin{figure}
   \centering
   \includegraphics[width=3cm]{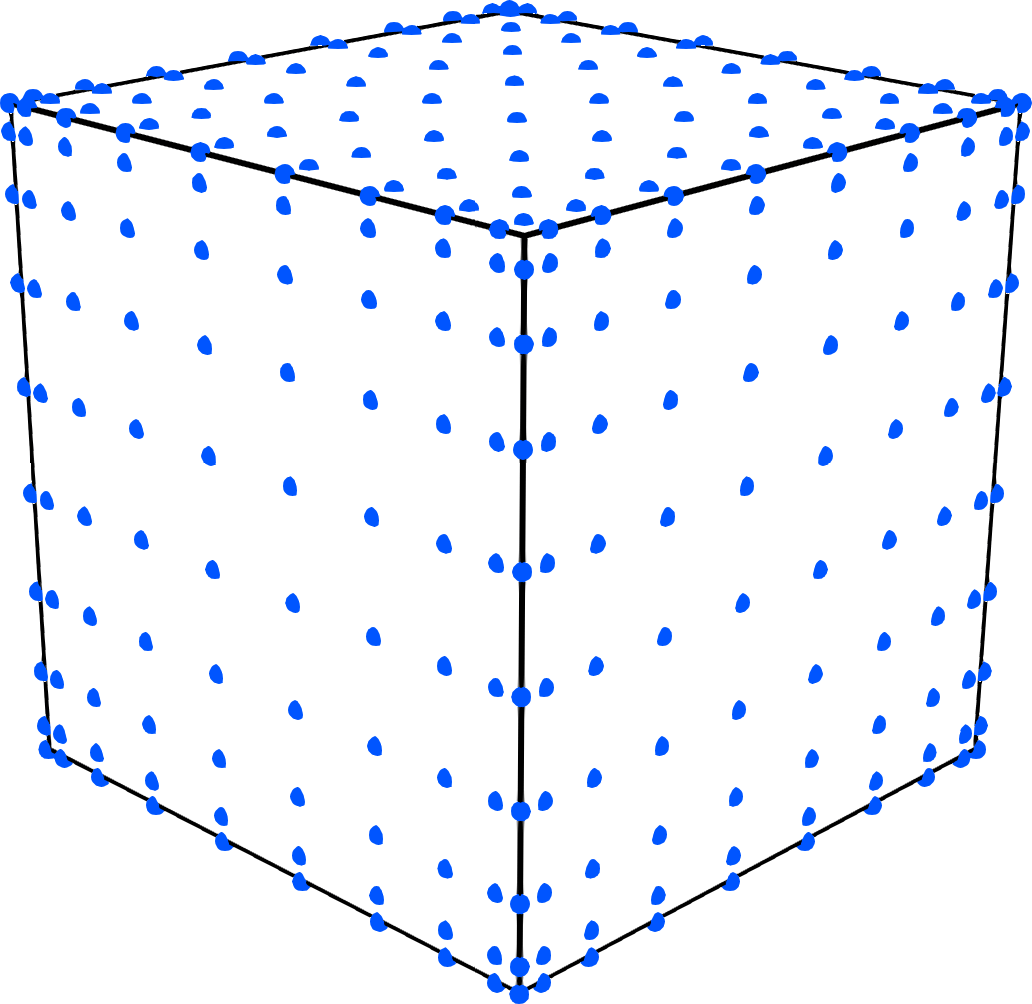}
   \hspace{1cm}
   \includegraphics[width=3cm]{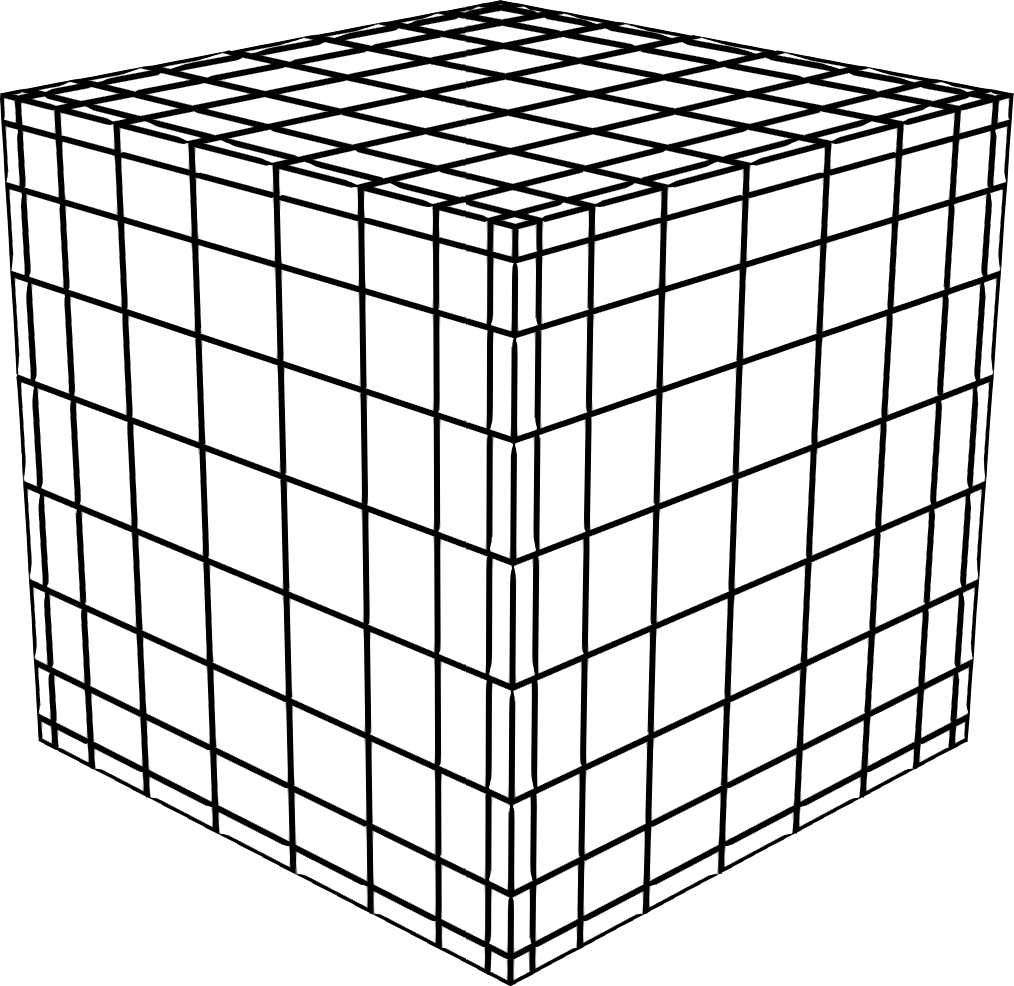}

   \caption{
      Left: high-order ($p=9$) hexahedral element with 10 Gauss--Lobatto nodes in each dimension.
      Right: subelement mesh with vertices at the Gauss--Lobatto points.
   }
   \label{fig:lor-cube}
\end{figure}
Given the above definitions, there is a natural identification of Raviart--Thomas degrees of freedom with subelement faces, and $L^2$ degrees of freedom with subelement volumes.
For our purposes, this choice of basis results in several important properties, the proofs of which can be found in \cite{Pazner2023}.
\begin{itemize}
   \item The Raviart--Thomas mass matrix $\M$ defined on $\Vh$ is spectrally equivalent to its diagonal, independent of mesh size and polynomial degree.
   \item The $L^2$ mass matrix $\W$ defined on $W_h$ is spectrally to its diagonal, independent of mesh size and polynomial degree.
   \item The discrete divergence matrix $\D : \Vh \to W_h$ is sparse, with two nonzeros per column (taking values $1$ and $-1$), regardless of polynomial degree. The matrix is entirely determined by the mesh topology (its entries are independent of mesh geometry or distortion), and it coincides exactly with the lowest-order discrete divergence matrix defined on the refined mesh (see the right panel of \Cref{fig:lor-cube}).
\end{itemize}

Because of the sparsity afforded by the discrete divergence matrix, it is advantageous to consider the transformed saddle-point matrix
\begin{equation}
   \label{eq:transformed-system}
   \AA :=
   \begin{pmatrix}
            \M_\beta & \D^\tr \\
            \D & -\W_\alpha^{-1}
   \end{pmatrix}
   =
   \begin{pmatrix}
      \I & \Zero \\
      \Zero & \W_\alpha^{-1}
   \end{pmatrix}
   \begin{pmatrix}
      \M_\beta & \B_\alpha^\tr \\
      \B_\alpha & -\W_\alpha
   \end{pmatrix}
   \begin{pmatrix}
      \I & \Zero \\
      \Zero & \W_\alpha^{-1}
   \end{pmatrix}
\end{equation}
using that $\D = \W_\alpha^{-1}\B_\alpha$ by the definitions in \eqref{eq:matrices}.
In contrast to the untransformed saddle-point system introduced in \eqref{eq:saddle-point}, the transformed system $\AA$ has much sparser off-diagonal blocks (whose sparsity, critically, is independent of the polynomial degree).
However, a downside of this transformation is that the matrix $\AA$ involves the \textit{inverse} of the $L^2$ mass matrix $\W_\alpha$.
In the context of an iterative solver such as MINRES, this means that the action of $\W_\alpha^{-1}$ will need to be computed at every iteration.
Because $W_h$ is a discontinuous space, $\W_\alpha$ is block-diagonal, and, at least for low orders, the blocks can be factorized efficiently element-by-element.
For higher orders, it is not feasible to factorize (or even assemble) the blocks of $\W_\alpha$, and so we turn to iterative matrix-free methods to apply the action of $\W_\alpha^{-1}$.
These methods are described in greater detail in \Cref{sec:dg-mass-inv}, and their performance is studied in the numerical results in \Cref{sec:dg-mass-gpu}.
The sparsity of the off-diagonal blocks and the diagonal spectral equivalence properties enumerated above allow for the sparse approximation of the Schur complement and efficient algebraic multigrid preconditioning, enabling the block preconditioning techniques described in the following section.

\subsection{Block preconditioners for saddle-point systems}
\label{sec:block-precond}

At this point, we make a brief digression to state some results concerning block preconditioning of $2 \times 2$ saddle-point systems.
We do not claim novelty of these results, however we collect them here in a form that will be convenient to reference later (see, e.g., \cite{Benzi2005,Zulehner2011,Pestana2015} for related results concerning block preconditioning of saddle-point systems).
Consider the saddle-point system
\begin{equation}
   \label{eq:generic-saddle-point}
   \AA = \begin{pmatrix}
      \A & \B^\tr \\
      \B & -\C
   \end{pmatrix}.
\end{equation}
While the particular saddle-point system of interest is that given by \eqref{eq:transformed-system}, in this section we only make the assumption that the blocks $\A$ and $\C$ in \eqref{eq:generic-saddle-point} are symmetric and positive-definite (but make no other assumptions about the particular form of the blocks).
The convergence of MINRES applied to this system can be estimated using bounds on the positive and negative parts of the spectrum,
\[
   \sigma(\AA) \subseteq [\mu_1, \mu_2] \cup [\nu_1, \nu_2]
   \quad \text{with} \quad
   \mu_1, \mu_2 < 0, \quad
   \nu_1, \nu_2 > 0.
\]
In particular, we are interested in the ratio $C/c$, where
\[
   C = |\lambda|_{\max} = \max \{ |\lambda| : \lambda \in \sigma(\AA) \}, \qquad
   c = |\lambda|_{\min} = \min \{ |\lambda| : \lambda \in \sigma(\AA) \},
\]
resulting in
\[
   c (\xx, \xx) \leq | (\AA\xx, \xx) | \leq C (\xx, \xx).
\]
Given a symmetric positive-definite preconditioner $\BB$, the relevant bounds are those that satisfy
\begin{equation}
   \label{eq:spectral-bounds}
   c (\BB \xx, \xx) \leq | (\AA\xx, \xx) | \leq C (\BB \xx, \xx),
\end{equation}
which can then be used to estimate the condition number, $\kappa(\BB^{-1}\AA) \leq C/c$.

\begin{proposition}
   \label{prop:diag-precond}
   Let $\BB$ be the block-diagonal preconditioner
   \[
      \BB = \begin{pmatrix}
         \A & \Zero \\
         \Zero & \S
      \end{pmatrix},
   \]
   where $\S = \C + \D \A^{-1} \D^\tr$ is the (negative) Schur complement of $\AA$ with respect to the $(1,1)$-block.
   This preconditioner is optimal in the sense that
   \[
      \sigma(\BB^{-1}\AA) \subseteq
         \left[-1, (1 - \sqrt{5})/2 \right) \cup
         \left[1,  (1 + \sqrt{5})/2 \right)
      \ \text{and}\ %
      \kappa(\BB^{-1}\AA) \leq \frac{\sqrt{5} + 1}{\sqrt{5} - 1} \approx 2.618\ldots
   \]
\end{proposition}
\begin{proof}
   Let $\lambda$ be an eigenvalue of $\BB^{-1}\AA$, i.e.
   \begin{align*}
      \A \x + \B^\tr \y &= \lambda \A \x, \\
      \B \x - \C \y &= \lambda \S \y,
   \end{align*}
   for some $(\x, \y) \neq 0$.
   Multiplying the first equation by $\B\A^{-1}$ and eliminating $\x$, we obtain
   \begin{equation}
      \label{eq:lambda}
      \left( \lambda\C + (\lambda^2 - \lambda - 1) \S \right) \y = 0.
   \end{equation}
   Considering the case $\lambda \neq 1$, we have $\y \neq 0$ and so $(\S \y, \y) \geq (\C \y, \y) > 0$.
   Equation \eqref{eq:lambda} then implies that
   \[
      \lambda (\C \y, \y) + (\lambda^2 - \lambda - 1) (\S \y, \y) = 0,
   \]
   and so $\lambda$ and $(\lambda^2 - \lambda - 1)$ must have opposite signs.
   We consider the two cases:

   \begin{itemize}
      \item
         If $\lambda > 0$ then this implies $\lambda^2 - \lambda - 1 < 0$ and so $\lambda < (\sqrt{5} + 1) / 2$.
         On the other hand, $(\C\y, \y) \leq (\S\y, \y)$ implies $\lambda (\S\y, \y) + (\lambda^2 - \lambda - 1) (\S\y, \y) \geq 0$, and so $\lambda^2 \geq 1$, and $\lambda \in [1, (\sqrt{5} + 1)/2)$.
      \item
         If $\lambda < 0$ then $\lambda^2 - \lambda - 1 > 0$ and so  $\lambda < (1 - \sqrt{5})/2$.
         By the same reasoning as above, $\lambda^2 \leq 1$, so $\lambda \geq -1$, and $\lambda \in [-1, (1 - \sqrt{5})/2)$. \myqed
   \end{itemize}
\end{proof}

The condition number in \Cref{prop:diag-precond} can be improved slightly by using a different scaling of the blocks.

\begin{proposition}
   \label{prop:scaled-precond}
   The block-diagonal preconditioner
   \[
      \BB = \begin{pmatrix}
         2\A & \Zero \\
         \Zero & \S
      \end{pmatrix},
   \]
   is optimal in the sense that
   \[
      \sigma(\BB^{-1}\AA) \subseteq
         \left[-1, -1/2 \right) \cup
         \left[1/2, 1 \right)
      \qquad\text{and}\qquad
      \kappa(\BB^{-1}\AA) \leq 2.
   \]
   The constant 2 is the optimal relative scaling of $\A$ in terms of minimizing the bound on $\kappa(\BB^{-1} \AA)$.
\end{proposition}
\begin{proof}
   Let $\tau > 0$ be a scaling constant, and consider $\BB$ given by
   \[
      \BB = \begin{pmatrix}
         \tau \A & \Zero \\
         \Zero & \S
      \end{pmatrix}.
   \]
   Repeating the argument from the proof of \Cref{prop:diag-precond}, we see that either $\lambda = 1/\tau$ or
   \[
      \left( \tau \lambda \C + (\tau \lambda^2 - \lambda - 1) \S \right) \y = 0.
   \]
   Therefore, the eigenvalues are contained in the intervals where the quadratic polynomials $\tau \lambda^2 - \lambda - 1$ and $\tau \lambda^2 + (\tau - 1)\lambda - 1$ have opposite signs.
   The roots of these polynomials are $\lambda_1 = (1 - \sqrt{1 + 4\tau})/(2\tau)$, $\lambda_2 = (1 + \sqrt{1 + 4 \tau})/(2 \tau)$, $\lambda_1^* = 1/\tau$, and $\lambda_2^* = -1$, respectively.
   Note that $\lambda_2^* < \lambda_1 < 0 < \lambda_1^* < \lambda_2$, and so
   \[
      \kappa(\BB^{-1} \AA) \leq
         \frac{
            \max(|\lambda_2|, |\lambda_2^*|)
         }{
            \min(|\lambda_1|, |\lambda_1^*|)
         }.
   \]
   We consider three cases depending on the value of $\tau$:
   \begin{itemize}
      \item $\tau < 2$. Then, $|\lambda_2| > |\lambda_2^*|$ and $|\lambda_1| < |\lambda_1^*|$, resulting in the bound
      \[
         \kappa(\BB^{-1}\AA) \leq \frac{|\lambda_2|}{|\lambda_1|}
         = \frac{|1 + \sqrt{1 + 4\tau}|}{|1 - \sqrt{1+4\tau}|} > 2.
      \]
      \item $\tau > 2$. Then, $|\lambda_2| < |\lambda_2^*|$ and $|\lambda_1| > |\lambda_1^*|$, resulting in the bound
      \[
         \kappa(\BB^{-1}\AA) \leq \frac{|\lambda_2^*|}{|\lambda_1^*|}
         = \frac{1}{1/\tau} > 2.
      \]
      \item $\tau = 2$. Then $|\lambda_2| = |\lambda_2^*| = 1$ and $|\lambda_1| = |\lambda_1^*| = 1/2$, resulting in the best bound among these three cases,
      \[
         \kappa(\BB^{-1}\AA) \leq \frac{|\lambda_2|}{|\lambda_1|} = \frac{|\lambda_2^*|}{|\lambda_1^*|} = 2. \myqed
      \]
   \end{itemize}
\end{proof}

In the above, we are assuming that the applying the preconditioner $\BB$ involves exactly inverting the matrix $\A$ and Schur complement $\S$.
In practice this is typically infeasible, and so the inverses of these blocks are themselves approximated by spectrally equivalent preconditioners.
The following proposition shows that this does not impact the optimality of the preconditioners.

\begin{proposition}
   Given a saddle-point system $\AA$ and preconditioner $\BB$, suppose that the inequalities \eqref{eq:spectral-bounds} hold with constants $c$ and $C$.
   Suppose further that $\widetilde{\BB}$ is another preconditioner, spectrally equivalent to $\BB$, i.e.
   \begin{equation}
      \label{eq:spec-equiv-precond}
      \tilde{c} (\widetilde{\BB} \xx, \xx) \leq (\BB \xx, \xx) \leq \tilde{C} (\widetilde{\BB} \xx, \xx).
   \end{equation}
   Then, $\widetilde{\BB}^{-1} \AA$ is uniformly well-conditioned, with condition number bound
   \[
      \kappa(\widetilde{\BB}^{-1} \AA) \leq C \tilde{C} / (c \tilde{c}).
   \]
\end{proposition}
\begin{proof}
   This follows immediately by combining \eqref{eq:spectral-bounds} and \eqref{eq:spec-equiv-precond}.
\end{proof}

\begin{remark}
   As a consequence of the above results, a block-diagonal preconditioner of the form
   \[
      \BB = \begin{pmatrix}
         \tau \widetilde{\M} & \Zero \\
         \Zero & \widetilde{\S}
      \end{pmatrix},
   \]
   where $\widetilde{\M}$ is spectrally equivalent to $\M$ and $\widetilde{\S}$ is spectrally equivalent to the Schur complement $\S$, will result optimal (i.e.\ discretization parameter independent) convergence for MINRES when applied to the saddle-point system $\AA$.
   The relative scaling of the blocks determined by the coefficient $\tau$ should be chosen in accordance with \Cref{prop:scaled-precond}
\end{remark}

\begin{remark}
   Block-triangular preconditioners of the form
   \[
      \label{eq:block-triangular-exact}
      \BB = \begin{pmatrix}
         \M & \B^\tr \\
         \Zero & \S
      \end{pmatrix}
   \]
   result in a preconditioned system with exactly one eigenvalue, $\sigma(\BB^{-1}\AA) = \{ 1 \}$, and whose minimal polynomial has degree two, guaranteeing convergence of GMRES in at most two iterations \cite{Benzi2005}.
   In practice, the diagonal blocks must be replaced with spectrally equivalent approximations $\widetilde{\M}$ and $\widetilde{\S}$.
   In this setting, convergence estimates for GMRES may be obtained using a field of values analysis, where the resulting bounds depend on the inf-sup constant of the corresponding finite element problem \cite{Aulisa2018}.
   However, because the triangular preconditioners are non-symmetric, GMRES instead of MINRES must be used, increasing the memory requirements.
   Whether the reduced number of iterations warrants sacrificing the short-term recurrence is typically problem-specific;
   in the numerical results we consider only the block-diagonal preconditioner.
\end{remark}

\begin{remark}
   Using the so-called Bramble--Pasciak (BP) transformation, it is possible to use a block-triangular preconditioner with the conjugate gradient method in a transformed inner product \cite{Bramble1988}.
   This inner product requires a scaled approximation of the $(1,1)$-block of the matrix, typically necessitating the solution of an eigenvalue problem with the matrix $\widetilde{\M}^{-1}\M$.
   The total work per iteration of the BP--CG method is similar to that of block-diagonally preconditioned MINRES.
   Preliminary investigations into the performance of the BP--CG method applied to the problems considered presently indicate largely similar performance to that of MINRES with block-diagonal preconditioning (but with the added cost of estimating the extremal eigenvalue of $\widetilde{\M}^{-1}\M$).
\end{remark}

\subsection{Application to the grad-div system}
\label{sec:grad-div-precond}

We now apply the block preconditioning techniques developed in \Cref{sec:block-precond} to the mixed ($\Hdiv$ and $L^2$) saddle-point system $\AA$ in the form given by \eqref{eq:transformed-system}.
Using the properties of the interpolation-histopolation basis (see \Cref{sec:basis}), the $\Hdiv$ mass matrix $\M_\beta$ is spectrally equivalent to its diagonal, and so we can take $\widetilde{\M} = \diag(\M_\beta)$.
We now turn to approximating the Schur complement $\S = \W_\alpha + \D \M_\beta^{-1} \D^\tr$.
Using the spectral equivalence of $\W_\alpha$ to its diagonal, we take $\widetilde{\W} = \diag(\W_\alpha)$, and then define
\begin{equation}
   \label{eq:approx-schur}
   \widetilde{\S} = \widetilde{\W}^{-1} + \D \widetilde{\M}^{-1} \D^\tr.
\end{equation}
Because of the properties of the basis, the matrix $\widetilde{\S}$ takes a particularly simple structure.
Each degree of freedom in the space $\Vh$ corresponds to a \textit{subelement face}, and each degree of freedom in $W_h$ corresponds to a \textit{subelement volume}.
The connectivity pattern of the $\widetilde{\S}$ is determined by the topological connectivity of the subelement mesh.
In particular, let $\mathcal{F}(i)$ denote the set of Raviart--Thomas degrees of freedom $k$ corresponding to subelement faces belonging to the subelement volume $i$.
Given two subelement volumes (i.e., two $L^2$ degrees of freedom) sharing a common face, let $k_{ij}$ denote the degree of freedom corresponding to the shared face.
Then, the entries of $\widetilde{\S}$ are given by
\begin{equation}
   \label{eq:approx-schur-entries}
   \widetilde{\S}_{ij} =
   \begin{cases}
      \widetilde{\W}_{ii}^{-1} + \sum_{k \in \mathcal{F}(i)} \widetilde{\M}^{-1}_{kk} &\quad\text{if $i = j$,} \\
      -\widetilde{\M}^{-1}_{kk} &\quad\text{if $i$ and $j$ share common face $k = k_{ij}$,} \\
      0 &\quad\text{otherwise.}
   \end{cases}
\end{equation}
Consequently, the matrix $\widetilde{\S}$ is sparse, since the number of off-diagonal entries is bounded by the number of faces per element (4 for a quadrilateral mesh in 2D, 6 for a hexahedral mesh in 3D), independent of the polynomial degree of the finite element spaces.
Additionally, $\D\widetilde{\M}^{-1}\D^\tr$ can be viewed as a simple discretization of the Laplacian, and the following arguments show that $\S$ is well-preconditioned by methods such as algebraic multigrid.
\begin{proposition}
   The approximate Schur complement $\widetilde{\S}$ is an M-matrix.
\end{proposition}
\begin{proof}
   This follows from the explicit form \eqref{eq:approx-schur-entries} using that $\W_\alpha$ and $\M_\beta$ are positive-definite, and so $\widetilde{\W}^{-1}_{ii} > 0$ and $\widetilde{\M}^{-1}_{kk} > 0$.
\end{proof}
Furthermore, $\widetilde{\S}$ can be interpreted as a diagonal (approximate) mass matrix $\widetilde{W}$ plus a \textit{weighted graph Laplacian} term $\D \widetilde{\M}^{-1}\D^\tr$.
The graph Laplacian $\L$ of an edge-weighted graph is defined by
\begin{equation}
   \label{eq:graph-laplacian}
   \L_{ij} = \begin{cases}
      w_i &\quad \text{if $i = j$,} \\
      -w_{ij} &\quad \text{if $i \sim j$,} \\
      0 &\quad \text{otherwise,}
   \end{cases}
\end{equation}
where we write $i \sim j$ if the graph vertices $i$ and $j$ are connected by an edge (with weight $w_{ij}$). The vertex weight $w_i$ is defined by $w_i = \sum_{i \sim j} w_{ij}$.
Comparing \eqref{eq:approx-schur-entries} to \eqref{eq:graph-laplacian}, it is straightforward to see that $\D \widetilde{\M}^{-1}\D^\tr$ is the graph Laplacian corresponding to the weighted face-neighbor connectivity graph of the subelement mesh, with weights corresponding to the reciprocal of the diagonal entries of $\M_\beta$ (the diagonal entries of $\M_\beta$ scale like area of the corresponding subelement face).
Note that this system is very closely related to the lowest-order discontinuous Galerkin discretization described in \cite{Pazner2023}.
Algebraic multigrid preconditioners for linear systems of this form are well-studied, and can be expected to result uniform convergence independent of problem size.
The following proposition summarizes the application of block-diagonal preconditioning to the grad-div system.

\begin{proposition}
   Let $\widetilde{\M}$ denote the diagonal of $\M_\beta$ and let $\widehat{S}$ denote a matrix spectrally equivalent to $\S$ (for example, $\widehat{S}^{-1}$ could be given by one AMG V-cycle applied to $\widetilde{\S}$).
   Then, $\BB^{-1}\AA$ is uniformly well conditioned, where $\BB$ is the block diagonal preconditioner
   \[
      \BB = \begin{pmatrix}
         \widetilde{\M} & \Zero \\
         \Zero & \widehat{\S}
      \end{pmatrix}.
   \]
\end{proposition}

\subsection{Application to the Darcy system}
\label{sec:darcy}

The treatment of the Darcy system \eqref{eq:darcy-nonzero} is largely similar to that of the grad-div system \eqref{eq:saddle-point}.
The relevant difference is that in the case of the Darcy system, the off-diagonal blocks $\B$ and $\B^\tr$ are not weighted by the coefficient $\gamma$.
As a result, the transformed system takes the slightly modified form
\begin{equation}
   \label{eq:transformed-system-darcy}
   \AA' =
   \begin{pmatrix}
            \M_{1/\varepsilon} & \D^\tr \\
            \D & -\W^{-1}\W_\gamma\W^{-1}
   \end{pmatrix}
   =
   \begin{pmatrix}
      \I & \Zero \\
      \Zero & \W^{-1}
   \end{pmatrix}
   \begin{pmatrix}
      \M_{1/\varepsilon} & \B^\tr \\
      \B & -\W_\gamma
   \end{pmatrix}
   \begin{pmatrix}
      \I & \Zero \\
      \Zero & \W^{-1}
   \end{pmatrix}
\end{equation}
In the case that $\gamma v_h \in W_h$ for all $v_h \in W_h$, the $(2,2)$-block of $\AA'$ can be simplified to $-\W_{1/\gamma}^{-1}$.
This can be seen by noting that $\W^{-1} \W_\gamma \w$ (where the vector $\w$ corresponds to $w_h \in W_h$) is the unique element $u_h \in W_h$ satisfying
\[
   (u_h, v_h) = (\gamma w_h, v_h)
\]
for all $v_h \in W_h$.
Likewise, $\W_{1/\gamma}^{-1} \W \w$ is the unique element $u_h' \in W_h$ satisfying
\[
   (\gamma^{-1} u_h', v_h) = (w_h, v_h)
\]
for all $v_h \in W_h$.
When $\gamma v_h \in  W_h$, we have
\[
   (u_h, v_h) = (\gamma w_h, v_h) = (w_h, \gamma v_h) = (w_h, \gamma v_h) = (\gamma^{-1} u_h', \gamma v_h) = (u_h', v_h),
\]
and hence $\W^{-1}\W_\gamma$ = $\W_{1/\gamma}^{-1}\W$, and so $\W^{-1}\W_\gamma\W^{-1} = \W_{1/\gamma}^{-1}\W \W^{-1} = \W_{1/\gamma}^{-1}$.
In particular, this holds when the coefficient $\gamma$ is piecewise constant (i.e.\ $\gamma|_\k$ is constant for each $\k \in \T$).
In the general case, this simplification does not hold, and the full form of the $(2,2)$-block is used.
In the case of the Darcy problem with zero $(2,2)$-block (corresponding to the Poisson problem $-\nabla \cdot (\varepsilon \nabla p) = g$), the matrix $\W^{-1}$ is not required to compute the action of $\AA'$, reducing the computational cost of the MINRES iterations.

As before, the $\Hdiv$ mass matrix $\M_{1/\varepsilon}$ can be approximated by its diagonal, and $\W^{-1}\W_\gamma\W^{-1}$ can be approximated be the product of (the reciprocals of) the diagonals of the terms in the product.
The approximate Schur complement $\widetilde{\S}$ takes the same form \eqref{eq:approx-schur-entries} as in the preceding section, and can be effectively preconditioned with algebraic multigrid.

\subsection{\texorpdfstring{$L^2$}{L2} mass inverse}
\label{sec:dg-mass-inv}

Solving the transformed saddle-point system $\AA$ given by \eqref{eq:transformed-system} by a Krylov subspace method such as MINRES (or GMRES in the case of a non-symmetric preconditioner) requires, at every iteration, the action of the operator $\W^{-1}$.
The matrix $\W_\alpha$ corresponds to the $\alpha$-weighted $L^2$ inner product on the space $W_h$, i.e.,
\[
   \r^\tr \W_\alpha \q = (\alpha q, r), \quad q,r \in W_h.
\]
Since the space $W_h$ is discontinuous, $\W$ is block-diagonal, with blocks $\W_i$ corresponding to the $i$th element of the mesh $\k_i$.
The block $\W_i$ represents the $L^2$ inner product over $\k_i$,
\[
   \r_i^\tr \W_i \q_i = (\alpha q, r)_{\k_i},
\]
where $\q_i$ and $\r_i$ are subvectors of $\q$ and $\r$ consisting of those degrees of freedom belonging to the element $\k_i$.
Hence, each block $\W_i$ is symmetric and positive-definite, and has size $p^d \times p^d$.

For modest values of $p$, assembling and factorizing each of the blocks $\W_i$ (e.g. using Cholesky factorization or spectral decomposition) is an appropriate strategy, allowing for the application of $\W_\alpha^{-1}$ with roughly the same cost as standard operator evaluations or matrix-vector products.
However, for higher-order problems, it becomes infeasible to assemble, store, and factorize the blocks $\W_i$.
Each block requires $\mathcal{O}(p^{2d})$ storage, and the assembly and factorization require $\mathcal{O}(p^{3d})$ operations.
In this case, it is preferable to compute the action of $\W_\alpha^{-1}$ matrix-free using an iterative method.

In the matrix-free setting, the action of $\W_\alpha^{-1}$ is computed block-by-block using a local conjugate gradient (CG) iteration.
Because of properties of the basis for the space $W_h$, each block $\W_i$ is spectrally equivalent to its diagonal, and so using $\diag(\W_i)$ as a preconditioner will result in a number of CG iterations that is independent of the polynomial degree $p$.
Furthermore, the action of $\W_i$ can be computed efficiently with matrix assembly using a sum factorization approach \cite{Orszag1980,Pazner2018e}.
Although this method is asymptotically optimal in terms of iterations (the cost of computing $\W_\alpha^{-1}$ is proportional to the cost of evaluating $\W_\alpha$, independent of discretization parameters), the constants in the condition number estimates can be improved by using a change of basis.

In \Cref{fig:l2-mass} we show the condition number of the 2D and 3D $L^2$ mass matrix on a single highly skewed element.
In each case, we consider the diagonally scaled mass matrix using a different basis.
In all cases, the condition number is bounded with respect to $p$.
As predicted by analytical estimates (cf.\ \cite{Teukolsky2015,Kolev2022}) the condition number of the Gauss--Lobatto mass matrix \textit{improves} with increasing $p$.
The condition number obtained using the interpolation--histopolation basis displays a preasymptotic increase in condition number, and for larger $p$ results in largely similar condition numbers as with the Gauss--Lobatto basis.
The Gauss--Legendre nodal basis results in the best conditioned system, with maximal condition number of approximately 1.12 in these tests.
Note that on elements whose transformation has constant Jacobian determinant (i.e.\ parallelepipeds) the Gauss--Legendre mass matrix is diagonal (and so in this special case, diagonal preconditioning is exact).
However, in the cases we consider here, the elements are skewed, resulting in non-constant Jacobian determinants, and so the mass matrix is not diagonal with any of the bases considered.
Even in this case, the diagonally preconditioned matrix remains extremely well conditioned, with condition number only slightly larger than unity.

\begin{figure}
   \centering
   \includegraphics{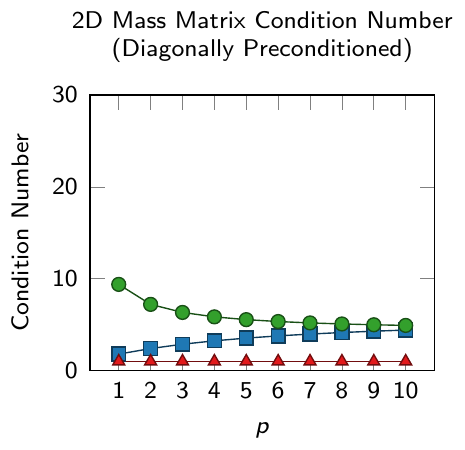}
   \includegraphics{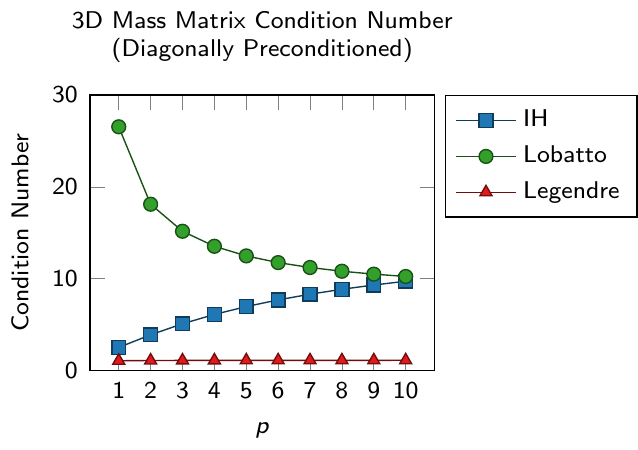}
   \caption{
      Condition number of the $L^2$ mass matrix preconditioned by diagonal scaling for different choices of basis.
      Left: 2D case on a skewed quadrilateral.
      Right: 3D case on a skewed hexahedron.
   }
   \label{fig:l2-mass}
\end{figure}

Motivated by these results, faster CG convergence can be obtained by first performing a change of basis, and then iteratively solving the system in the Gauss--Legendre basis with diagonal preconditioning.
Performing the change of basis necessitates the transformation of the right-hand side and the obtained solution.
If a non-zero initial guess is used, one additional transformation must also be performed.
Each of these transformations is generally less expensive than one operator evaluation, and so this will result in an overall efficiency gain if the total iteration count is reduced by at least two or three.
Numerical experiments studying these considerations are presented in \Cref{sec:dg-mass-gpu}.

\subsection{Notes on the untransformed system}

Instead of considering the transformed system $\AA$ defined by \eqref{eq:transformed-system}, it is also possible to apply block preconditioners directly to the saddle-point system \eqref{eq:saddle-point}.
This has the advantage that each MINRES iteration would require computing the action of $\W_\alpha$ rather than $\W_\alpha^{-1}$, potentially leading to computational savings.
However, this approach has two drawbacks.
Firstly, the off-diagonal blocks $\B_\alpha$ and $\B_\alpha^\tr$ do not enjoy the same favorable sparsity as $\D$ and $\D^\tr$, and therefore are more expensive to compute.
Secondly, and more importantly, the Schur complement of the original system takes the form
\begin{equation}
   \label{eq:untransformed-schur}
   \S' = \W_\alpha + \B_\alpha \M_\beta^{-1} B_\alpha^\tr
      = \W_\alpha + \W_\alpha \D \M_\beta^{-1} \D^\tr \W_\alpha.
\end{equation}
In order to construct block preconditioners, approximations of the inverse of $\S'$ are required.
As in the construction of the approximate Schur complement $\widetilde{\S}$ in \eqref{eq:approx-schur}, it is possible to replace $\W_\alpha$ and $\M_\beta$ by their diagonals, to obtain a matrix $\widetilde{\S}'$ spectrally equivalent to $\S'$.
This approximation of the second term on the right-hand side of \eqref{eq:untransformed-schur} is significantly worse that the approximation of the transformed Schur complement $\S$ by \eqref{eq:approx-schur}, since it involves the product of three approximations rather than one.
A very rough estimate would suggest that, while still asymptotically optimal, the condition number of $\widetilde{\S}'^{-1} \S'$ could be very roughly bounded by the cube of the condition number of $\widetilde{\S}^{-1}\S$.

In \Cref{fig:schur-cond}, we display computed condition numbers of the transformed preconditioned Schur complement $\widetilde{\S}^{-1}\S$ and the untransformed version $\widetilde{\S}'^{-1} \S'$ for increasing polynomial degree.
We consider both the unit cube $[0,1]^d$, and a skewed single-element mesh (the same mesh used in the results shown in \Cref{fig:l2-mass}).
Although both versions of the preconditioned Schur complement have $\mathcal{O}(1)$ condition number with respect to polynomial degree, the empirical results show that the transformed Schur complement results in significantly smaller condition number, in accordance with arguments described above.
These results justify the use of the transformed system \eqref{eq:transformed-system}, despite the increased cost per iteration incurred by the inverse $L^2$ mass matrix.

\begin{figure}
   \centering
   \includegraphics[scale=0.99]{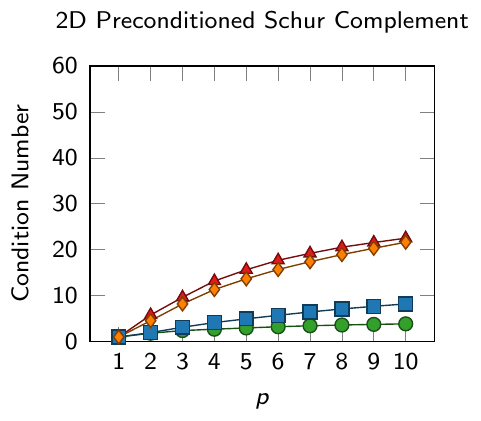}%
   \includegraphics[scale=0.99]{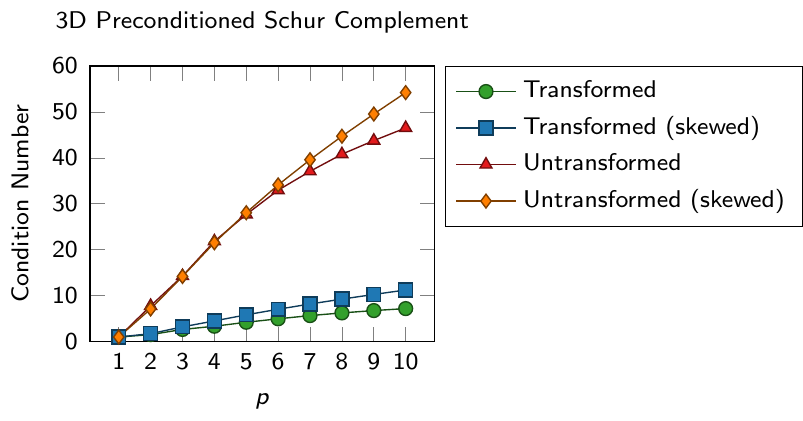}
   \caption{
      Computed condition numbers of preconditioner Schur complement.
      The transformed system is given by $\widetilde{\S}^{-1}\S$, and the untransformed system is given by $\widetilde{\S}'^{-1} \S'$.
   }
   \label{fig:schur-cond}
\end{figure}

\section{Algorithmic details and GPU acceleration}
\label{sec:alg-gpu}

In this section, we give a brief overview of the solution algorithm and describe the strategies used for GPU acceleration.
To solve the variational problem \eqref{eq:block-variational}, we solve the saddle-point system
\begin{equation}
   \label{eq:transformed-saddle-point}
   \begin{pmatrix}
      \M_\beta & \D^\tr \\
      \D & -\W_\alpha^{-1}
   \end{pmatrix}
   \begin{pmatrix} \u \\ \widetilde{\q} \end{pmatrix}
   =
   \begin{pmatrix} \f \\ \Zero \end{pmatrix},
\end{equation}
such that $(\u, \q)$ for $\q = \W_\alpha^{-1} \widetilde{\q}$ gives the solution to the untransformed system \eqref{eq:saddle-point}.
A schematic of the solution procedure is shown in \Cref{fig:solver-diagram}.
This system is solved using MINRES (with the block-diagonal preconditioner) or GMRES (with block-triangular preconditioners).
Each iteration requires application of the operators $\M_\beta$, $\D$, $\D^\tr$ and $\W_\alpha^{-1}$, and application of the preconditioners $\widetilde{\M}^{-1}$ and $\widehat{\S}^{-1}$.
The high-order mass matrix $\M_\beta$ can be applied efficiently using matrix-free algorithms with sum factorization (cf.~\cite{Abdelfattah2021,Fischer2020}).
The discrete divergence operator $\D$ can be represented explicitly as a sparse matrix (where the number of nonzeros per row remains bounded, independent of the polynomial degree).
The efficient application of $\W_\alpha^{-1}$ with GPU acceleration is discussed in \Cref{sec:dg-mass-gpu}.
The preconditioner $\widetilde{\M}^{-1}$ for the $(1,1)$-block is simply diagonal scaling, and is trivially parallelizable.
In this work, we use algebraic multigrid for the approximate Schur complement preconditioner; details are discussed in \Cref{sec:schur-gpu}.

\begin{figure}
   \centering
   \begin{tikzpicture}[%
      every node/.style={draw=black, anchor=west, font={\sf\small}},
      grow via three points={
         one child at (1.8,-0.7) and
         two children at (1.8,-0.7) and (1.8,-1.4)},
      edge from parent path={([xshift=1cm] \tikzparentnode.south west) |- (\tikzchildnode.west)},
      growth parent anchor=west,
   ]
      \node {Krylov method: MINRES or GMRES}
      child {node {Matrix-free operator evaluation}
         child {node {$\M_\beta$: partial assembly}}
         child {node {$\D$, $\D^\tr$: assembled sparse matrix}}
         child {node {$\W_\alpha^{-1}$: local CG in Gauss--Legendre basis}}
      }
      child [missing] {} child [missing] {} child [missing] {}
      child {node {Preconditioner application}
         child {node {$\widetilde{\M}^{-1}$: diagonal scaling}}
         child {node {$\widehat{\S}^{-1}$: AMG V-cycle}}
      };
   \end{tikzpicture}
   \caption{Solver diagram for block preconditioners applied to the transformed saddle-point system \eqref{eq:transformed-saddle-point}.}
   \label{fig:solver-diagram}
\end{figure}
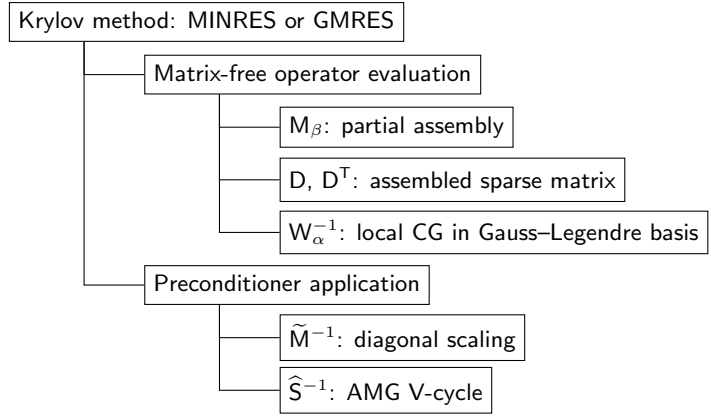

\subsection{GPU acceleration of the DG mass inverse}
\label{sec:dg-mass-gpu}

Because each iteration involving the system \eqref{eq:transformed-saddle-point} requires computing the action of $\W_\alpha^{-1}$, it is imperative to develop fast algorithms for the action of the inverse of the DG mass matrix.
As described in \Cref{sec:dg-mass-inv}, we consider several approaches, depending on the polynomial degree and spatial dimension.
For relatively small polynomial degrees, it is feasible to assemble and store the $L^2$ mass matrix, allowing for the use of direct methods;
for higher polynomial degrees, the memory requirements and computational time required to assemble the matrix render this option infeasible, necessitating the use of iterative solvers.

Since this matrix is block-diagonal, in the case where the matrix may be assembled, each block is stored as a dense matrix, which can then be factorized.
These operations are particularly well-suited for use with batched linear algebra libraries \cite{Dongarra2017}.
The performance of batched LU solvers and matrix-matrix multiplication operations on GPUs has been well-studied \cite{Abdelfattah2016,Villa2013}.
Presently, we consider two approaches.
The first approach is to perform a batched Cholesky factorization (\texttt{potrf}) during the setup, and then each application of the mass inverse will perform a triangular solve (\texttt{potrs}).
Batched versions of these routines are available on Nvidia and AMD GPU hardware through the cuSOLVER and rocSOLVER libraries.
The second approach computes the explicit inverse of each block (\texttt{getrf} followed by \texttt{getri}), and then each application is computed using a batched matrix-vector product (\texttt{gemv}).
While the computation of the inverse matrix necessitates additional overhead and is less numerically stable than solving the system using the Cholesky factorization, it has the advantage that the application of the inverse is computed using a matrix-vector product rather than a triangular solve, which is better-suited for fine-grained parallelism.
For symmetric positive-definite matrices of small sizes, matrix inversion typically possesses satisfactory stability bounds \cite{Croz1992}.
Since direct methods are only relevant for small-sized matrices in this application, the numerical stability of explicit matrix inverses is not a primary concern.

For larger polynomial degrees, direct methods become impractical or infeasible.
In this case, we perform conjugate gradient iterations locally on each element in parallel.
Because the system is block-diagonal, each conjugate gradient solver is independent, and no global reductions are required.
In our implementation, the entire CG iteration is performed in a single (fused) kernel, with one block of threads per element.
The number of threads per block is chosen to facilitate the application of the local mass operator.
This action corresponds to the ``\textit{CEED BP1}'' benchmark problem, and highly optimized GPU implementations for this operation are available \cite{Abdelfattah2016,Dobrev2017}.
This kernel also performs the initial change of basis operations, allowing the CG iteration to solve the better-conditioned Gauss-Legendre mass matrix, regardless of the user's choice basis for the discretization (see \Cref{fig:l2-mass}).
The vector operations (\texttt{axpy}) in the CG method are local to each element, and are trivially parallelizable.
The Dot products required by CG are performed using parallel reductions in each block.

\begin{figure}
   \centering
   \includegraphics[width=0.33\linewidth]{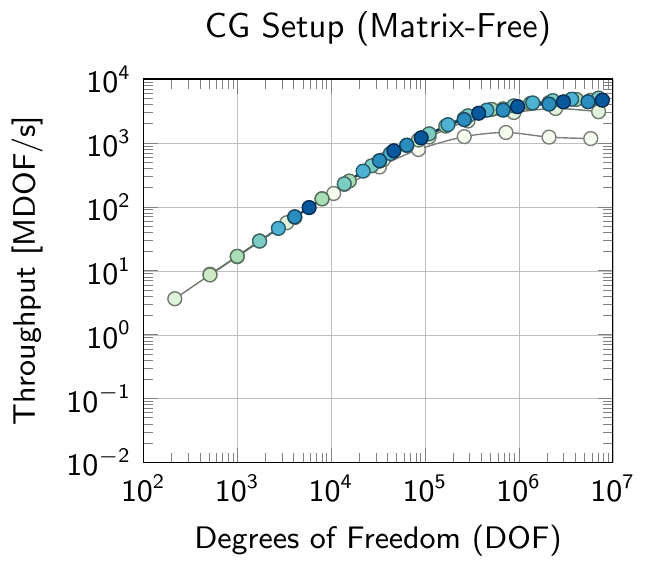}%
   \includegraphics[width=0.33\linewidth]{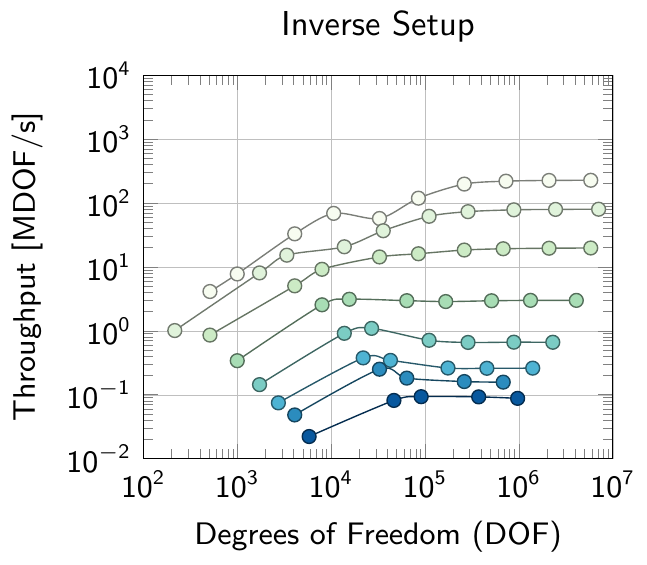}%
   \includegraphics[width=0.33\linewidth]{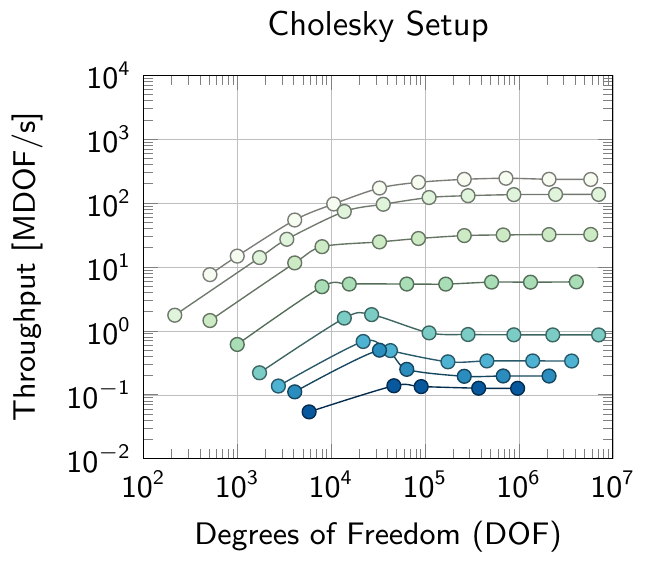}\\[8pt]
   \includegraphics[scale=0.85]{fig/dg_mass_inv/throughput_legend}

   \caption{GPU DG mass inverse setup phase.}
   \label{fig:dg-mass-inv-setup}
\end{figure}

In \Cref{fig:dg-mass-inv-setup}, we display GPU kernel throughput for the mass solver setup operations using these approaches.
The matrix-free CG setup consists of precomputing computing data (such as geometric factors) at quadrature points and assembling the diagonal of the matrix to be used for Jacobi preconditioning.
The setup for the direct methods requires assembling the local mass matrices and performing the dense factorizations.
The Cholesky and explicit inverse methods exhibit largely similar performance: the highest performance is achieved for the lowest-order discretizations (i.e.\ larger batches of smaller matrices).
Performance decreases monotonically as the polynomial degree (and matrix size) increase; large problems at $p \geq 5$ could not be run to completion because of memory limitations.
On the other hand, the CG setup, including assembly of the diagonal, lends itself to sum factorization, requiring $\mathcal{O}(p^{d+1})$ operations and $\mathcal{O}(p^d)$ memory, leading to significantly improved performance for higher polynomial degrees.

\begin{figure}
   \centering
   \setlength{\tabcolsep}{0 pt}
   \begin{tabular}{cc}
      \includegraphics[scale=0.85]{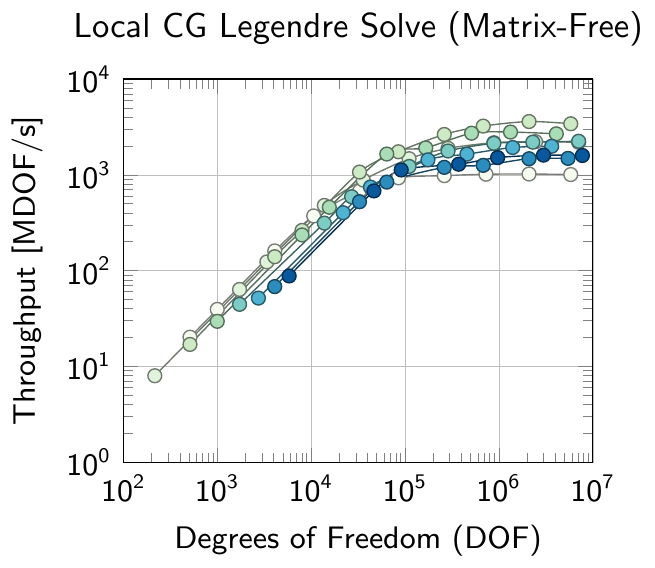} &
      \includegraphics[scale=0.85]{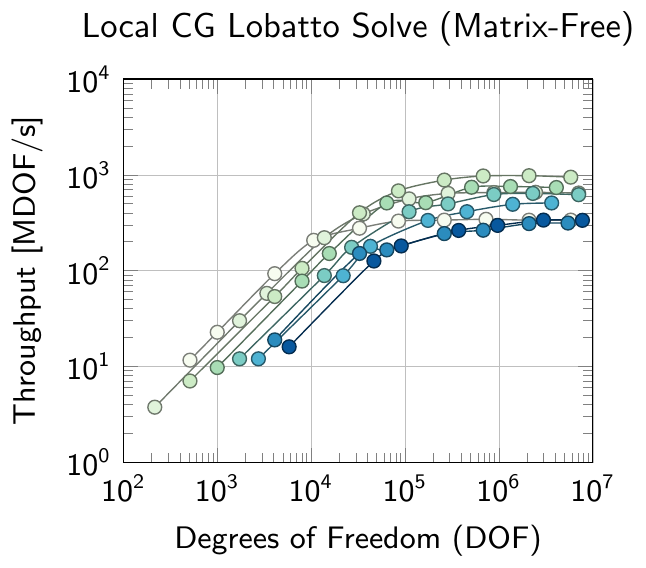} \\
      \includegraphics[scale=0.85]{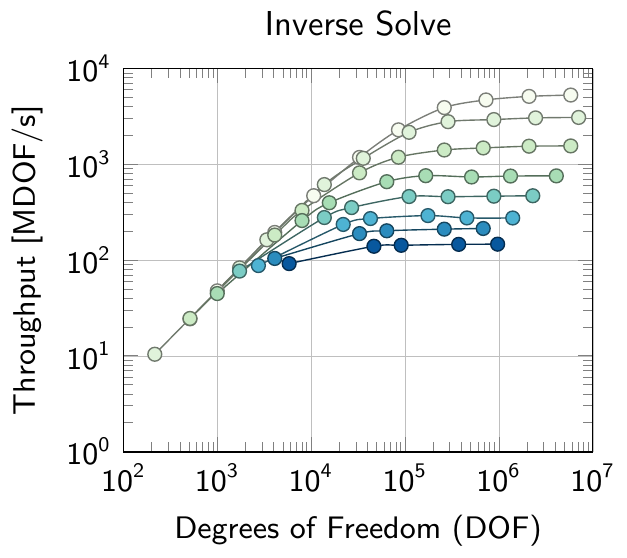} &
      \includegraphics[scale=0.85]{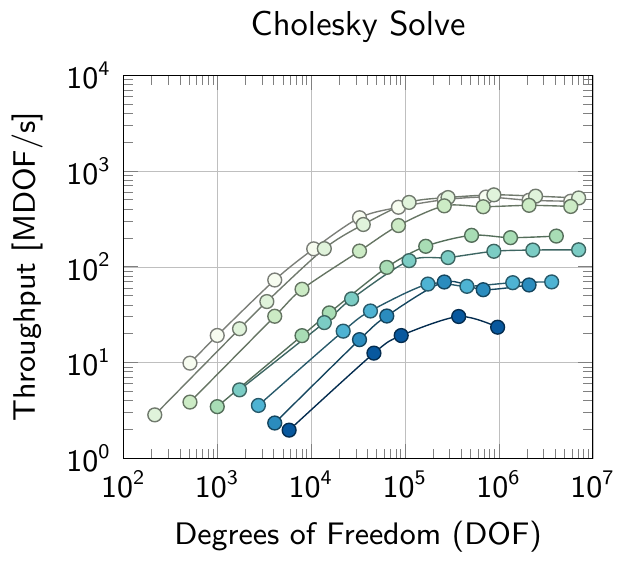} \\
   \end{tabular}
   \includegraphics[scale=0.85]{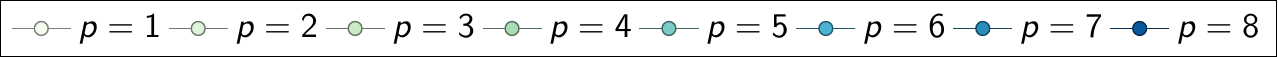}

   \caption{GPU DG mass inverse solve phase.}
   \label{fig:dg-mass-inv-solve}
\end{figure}

In the context of the iterative solvers considered presently, the setup time is of lesser concern than the operator application time:
the setup is performed once, while the action of $\W^{-1}$ must be computed at least once every iteration (and twice every iteration in the case of variable-coefficient Darcy problems, see \Cref{sec:darcy}).
\Cref{fig:dg-mass-inv-solve} displays GPU kernel throughput for computing the action of $\W^{-1}$ using the methods described in this section.
These results make clear that significantly higher performance is achieved in the CG solvers by first changing basis to the Gauss--Legendre basis, and then solving the better-conditioned system (cf.~\Cref{fig:l2-mass}).
The matrix-free methods using the Gauss--Legendre basis consistently achieve over $2 \times 10^9$ degrees of freedom per second, with a peak performance of $3.6 \times 10^9$ at $p=3$.
In contrast, a standard CG iteration with global reductions and without fused kernels achieves peak performance about 2--3 orders of magnitude slower than the local CG approach described here.

Comparing the direct solvers, it is clear that using the precomputed explicit inverse matrix results in significantly faster performance than the Cholesky factorization.
This is likely attributable to the increased fine-grain parallelism possible when performing batched matrix-vector products compared with triangular solves.
For $p=1$ and $p=2$, the explicit inverse results in the fastest solve time of all methods considered;
the highest throughput observed was $5.2\times10^9$ degrees of freedom per second, achieved for the largest problem size with $p=1$.
For $p=3$, the explicit inverse has throughput roughly equal to that of the Gauss--Legendre CG iteration (with significantly more expensive setup costs), and for $p \geq 4$, the Gauss--Legendre CG iteration is the most performant method considered.

\begin{table}
   \caption{
      DG mass inverse setup and solve run times.
      Solve time is reported for 100 applications of the inverse.
      The fastest method for each polynomial degree is shown in bold.
      For $p=8$, there was insufficient GPU memory to compute the explicit inverse, indicated by ``---''.
   }
   \label{tab:dg-mass-inv}
   \centering
   \newcommand{\bftab}{\fontseries{b}\selectfont}
   \sisetup{detect-weight=true,mode=text}
   \newcolumntype{Y}{S[table-format=2.4]}
   \ifsiam\resizebox{\linewidth}{!}{%
   \else\fi%
   \begin{tabular}{clYYYYYYYY}
      \toprule
      & & {$p = 1$} & {$p = 2$} & {$p = 3$} & {$p = 4$} & {$p = 5$} & {$p = 6$} & {$p = 7$} & {$p = 8$} \\
      \midrule
      \multirow{3}{*}{CG} & Setup & \bftab 0.0016 & \bftab 0.0007 & \bftab 0.0004 & \bftab 0.0004 & \bftab 0.0004 & \bftab 0.0004 & \bftab 0.0005 & \bftab 0.0004 \\
       & Solve {\footnotesize$\times 100$} & 0.1937 & 0.0841 & \bftab 0.0556 & \bftab 0.0598 & \bftab 0.0697 & \bftab 0.0770 & \bftab 0.0981 & \bftab 0.0845 \\
       & Total & 0.1953 & 0.0847 & \bftab 0.0561 & \bftab 0.0603 & \bftab 0.0701 & \bftab 0.0774 & \bftab 0.0986 & \bftab 0.0850 \\
      \midrule
      \multirow{3}{*}{Inverse} & Setup & 0.0071 & 0.0197 & 0.0782 & 0.4544 & 2.2950 & 5.8911 & 10.0749 & {---} \\
       & Solve {\footnotesize$\times 100$} & \bftab 0.0383 & \bftab 0.0598 & 0.1149 & 0.2296 & 0.3690 & 0.6095 & 0.8069 & {---} \\
       & Total & \bftab 0.0454 & \bftab 0.0796 & 0.1932 & 0.6840 & 2.6640 & 6.5007 & 10.8818 & {---} \\
      \midrule
      \multirow{3}{*}{Cholesky} & Setup & 0.0076 & 0.0128 & 0.0539 & 0.3077 & 1.9638 & 4.9905 & 8.8000 & 12.4510 \\
       & Solve {\footnotesize$\times 100$} & 0.3466 & 0.3182 & 0.3996 & 0.8135 & 1.1570 & 2.3312 & 2.5846 & 5.9352 \\
       & Total & 0.3542 & 0.3310 & 0.4535 & 1.1212 & 3.1208 & 7.3216 & 11.3846 & 18.3862 \\
       \bottomrule
   \end{tabular}
   \ifsiam}\else\fi
\end{table}

We conclude with a performance comparison on a series of problems with roughly $1.7 \times 10^6$ DOFs, with polynomial degree $p = 1, 2, \ldots, 8$.
In order to benchmark the solvers in a situation relevant to the saddle-point solvers considered in this paper, on each problem we apply the inverse mass matrix 100 times.
In \Cref{tab:dg-mass-inv}, we present the time required to construct the inverse (or set up the matrix-free CG), the time required for 100 inverse applications, and the total time.
The shortest runtime for each polynomial degree is shown in bold.
The matrix-free CG method has the fastest setup time among the methods considered, with roughly constant time required for all polynomial degrees (the longest setup time was required for $p=1$, which is consistent with the lower throughput results shown in \Cref{fig:dg-mass-inv-setup}).
The explicit inverse resulted in the fastest overall runtimes for $p=1$ ($4.3\times$ faster than CG) and $p=2$ ($1.1\times$ faster than CG).
For polynomial degrees $p \geq 3$, the matrix-free CG resulted in the fastest runtime;
the observed speed-up increases dramatically with the polynomial degree.
For $p=8$, there was insufficient memory to compute the explicit inverse.
For this test case, the CG setup time was $30{,}000\times$ faster than the Cholesky factorization, the solve time was $70\times$ faster than the triangular solve, resulting in an overall speed-up of $216\times$.
The Cholesky factorization was the slowest of the three methods considered;
the slightly faster factorization time compared with the explicit inverse was outweighed by the slower inverse applications.

Based on these results, in the contexts where the inverse of a fixed mass matrix will be applied many times (such as the iterative solvers considered presently, or in the case of time stepping on a fixed mesh), it would be efficient to use explicit inverses for $p \leq 2$, and matrix-free CG for $p \geq 3$.
If, on the other hand, the inverse of the mass matrix will be applied only once (for example, in situations involving a moving mesh, necessitating the recomputation of the mass matrix at every time step), the faster setup time suggests that the matrix-free CG iteration should be used even for $p \leq 2$.

\subsection{Schur complement preconditioning}
\label{sec:schur-gpu}

We now consider GPU-accelerated preconditioning of the approximate Schur complement $\widetilde{\S} = \widetilde{\W}^{-1} + \D \widetilde{M}^{-1} D^\tr$, where we recall that $\widetilde{\W}$ is the diagonal of the $L^2$ mass matrix, $\widetilde{M}$ is the diagonal of the $\Hdiv$ mass matrix, and $\D$ is the discrete divergence matrix.
The approach taken presently is algebraic multigrid preconditioning, and so the first step is the assembly of the matrix.
The diagonals $\widetilde{\W}$ and $\widetilde{\M}$ can be assembled efficiently using sum factorization \cite{Ronquist1987}.
Because of the properties of the interpolation--histopolation basis, the discrete divergence matrix $\D$ has the same structure as the lowest-order divergence matrix on a refined mesh.
In particular, the entries $\D_{ij}$ are given by
\[
   \D_{ij} = \begin{cases}
      \sigma_{ij} &\quad \text{if}\ j \in \mathcal{F}(i),\\
      0 &\quad\text{otherwise},
   \end{cases}
\]
where $\mathcal{F}(i)$ denotes the set of subelement faces adjacent to the subelement volume $i$ (cf.~\Cref{sec:grad-div-precond}), and $\sigma_{ij} = \pm 1$ denotes the orientation of the Raviart-Thomas basis function on subelement face $j$ relative to subelement volume $i$.
This matrix can be constructed directly in CSR format efficiently and in parallel on the GPU.
The algorithm used to construct this matrix is given in \Cref{alg:discrete-divergence}.
The kernel throughput of this algorithm is shown in the left panel of \Cref{fig:schur-throughput}.

\begin{algorithm}
   \caption{Construction of the discrete divergence matrix $\D$ in CSR format.}
   \label{alg:discrete-divergence}
   \algrenewcommand\algorithmicforall{\textbf{parallel for}}
   \algrenewtext{EndFor}{\textbf{end}}
   \begin{algorithmic}[0]
      \State $N \gets \text{\# of $L^2$ DOFs}$
      \ForAll{$i \in \{0, \ldots, N-1\}$}
         \Comment{For each $L^2$ DOF $i$ (subelement volume)}
         \State $I[i] \gets 2di$ \Comment $2d$ nonzeros per row
         \State $e \gets \texttt{element\_map}[i]$
            \Comment{$e$ is the mesh element containing DOF $i$}
         \State $i_{\textit{loc}} \gets \texttt{l2\_global\_to\_local}[i]$
            \Comment{$i_{\textit{loc}}$ is the local $L^2$ DOF index of $i$ within $e$}
         \For{$k \in \{1, \ldots, 2d\}$} \Comment{For each subelement face adjacent to $i$}
            \State $j_{\textit{loc}} \gets \texttt{volume\_to\_face}[k,i_{\textit{loc}}]$
            \Comment{$j_{\textit{loc}}$ is the local RT DOF index of the subelement face $k$}
            \State $\sigma_{\textit{loc}} \gets \texttt{local\_orientation}[j_{\textit{loc}}]$
            \Comment $\sigma_{\textit{loc}}$ is the orientation of DOF $j_{\textit{loc}}$ in the reference element
            \State $j \gets \texttt{rt\_local\_to\_global}[e, j_{\textit{loc}}]$
            \Comment{$j$ is the global DOF index}
            \State $\sigma_{\textit{glob}} \gets \texttt{global\_orientation}[j]$
            \Comment $\sigma_{\textit{loc}}$ is the orientation of the $j$ relative to $j_{\textit{loc}}$
            \State $J[2di + k] \gets j$
            \State $A[2di + k] \gets \sigma_{\textit{loc}}\sigma_{\textit{glob}}$
         \EndFor
      \EndFor
      \State $I[N] \gets 2dN$
   \end{algorithmic}
\end{algorithm}

\begin{figure}
   \centering
   \includegraphics[scale=0.85]{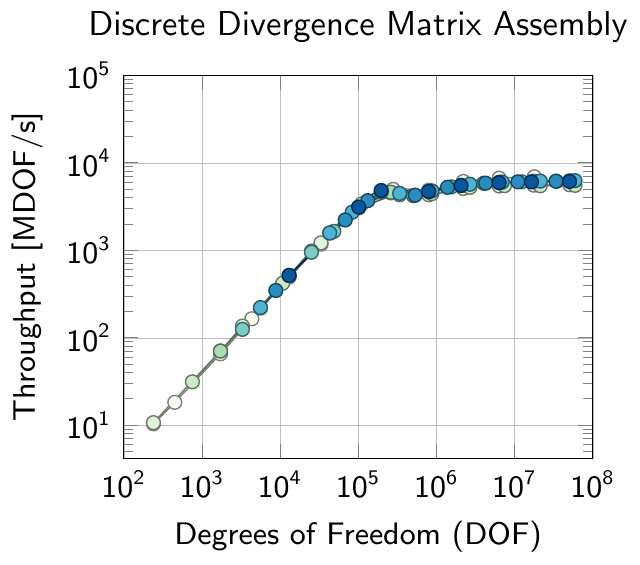}
   \includegraphics[scale=0.85]{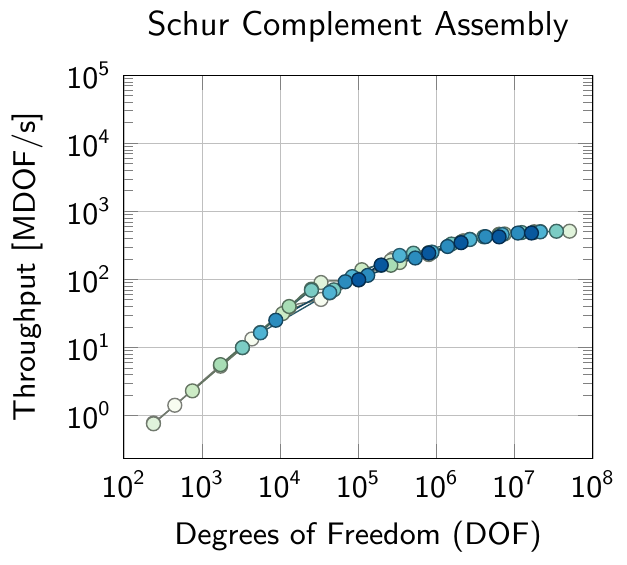}

   \includegraphics[scale=0.85]{fig/dg_mass_inv/throughput_legend}

   \caption{Throughput for discrete divergence and approximate Schur complement assembly.}
   \label{fig:schur-throughput}
\end{figure}

After constructing the discrete divergence, the approximate Schur complement can be constructed by computing the sparse matrix-matrix product $\D \widetilde{\M}^{-1} \D^\tr$.
In the current implementation, this product is computed using the sparse triple-product algorithms provided by the \textit{hypre} solvers library \cite{Falgout2021}.
The throughput for the triple product operation is shown in the right panel of \Cref{fig:schur-throughput}.
Given the assembled Schur complement in CSR format, the approximate inverse $\widehat{\S}^{-1}$ is given by one V-cycle of an algebraic multigrid preconditioner;
in this work, we use the BoomerAMG preconditioner from \textit{hypre} \cite{Henson2002}.
Detailed GPU performance studies for AMG preconditioners applied to similar problems are given in \cite{LORGPU2022}.

\section{Numerical results}
\label{sec:results}

In the following sections, we test the performance of the saddle-point solvers on a number of grad--div and Darcy test problems.
GPU results were performed on LLNL's \textit{Lassen} supercomputer, with four Nvidia V100 GPUs per node.
The methods were implemented in the framework of the open-source MFEM finite element library \cite{Anderson2020}, see also \url{https://mfem.org}.
A relative tolerance of $10^{-12}$ was used as the stopping criterion for all iterative solvers, unless specifically noted otherwise.
The BoomerAMG algebraic multigrid preconditioner from the \textit{hypre} library is used to precondition the approximate Schur complement \cite{Henson2002}, and \textit{hypre}'s ADS preconditioner applied to the low-order-refined system is used as a comparison \cite{Kolev2012}.
The AMG preconditioner for the Schur complement uses parallel maximal independent set (PMIS) coarsening with no aggressive coarsening levels.

\subsection{Grad-div problem: crooked pipe}

As a first numerical example, we consider the ``crooked-pipe'' grad-div problem.
This problem, which serves as a model problem related radiation diffusion simulations, is posed on a cylindrical sector.
The elements are divided into two materials;
the interface between the materials consists of highly stretched, anisotropic elements, increasing the problem difficulty.
We solve the problem
\[
   \nabla (\alpha \nabla \cdot \bm u) - \beta \bm u = \bm f,
\]
where the coefficients $\alpha$ and $\beta$ are piecewise constants determined by the materials.
A schematic of this problem is shown in \Cref{fig:crooked-pipe}.
In the outer region (colored green in \Cref{fig:crooked-pipe}), we have $\alpha = 1.88 \times 10^{-3}$ and $\beta = 2000$, and in the inner region (colored blue in \Cref{fig:crooked-pipe}), we have $\alpha = 1.641$ and $\beta = 0.2$.
This problem, first proposed in \cite{Graziani2000}, was considered as a benchmark problem using the ADS solver in \cite{Kolev2012}, using algebraic hybridization in \cite{Dobrev2019}, and with low-order-refined preconditioning in \cite{Pazner2023}.

\begin{figure}
   \centering
   \raisebox{-0.5\height}{\includegraphics[scale=0.7]{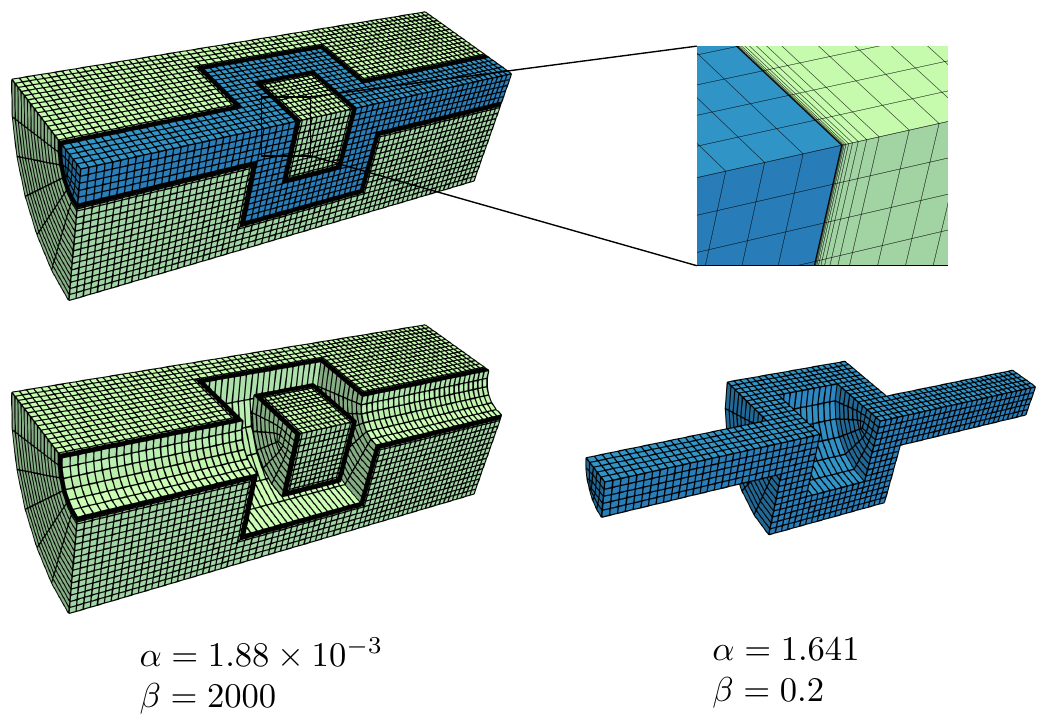}}
   \raisebox{-0.5\height}{\includegraphics[width=1.5in]{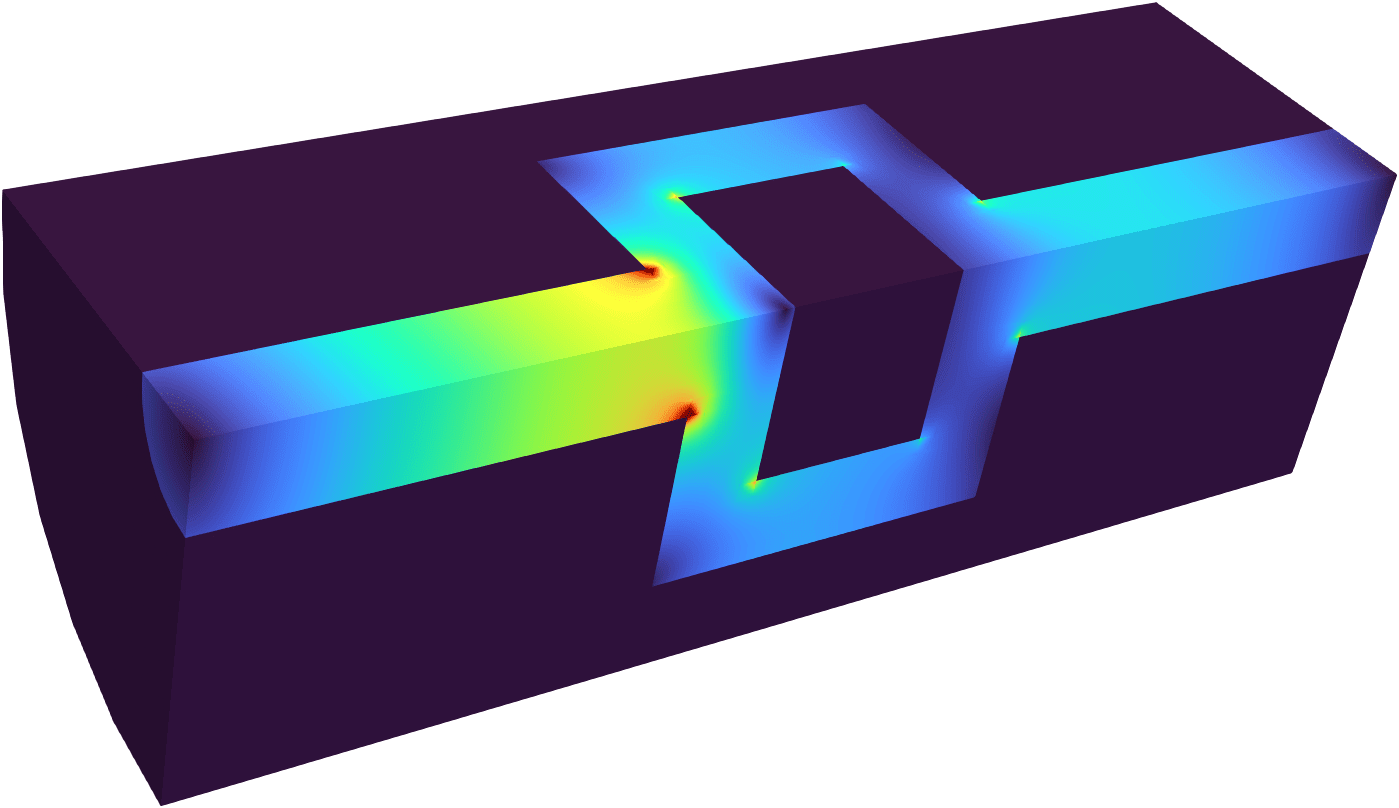}}
   \caption{
      Crooked pipe grad-div problem (figure reproduced from \cite{Pazner2023}).
      Left panel: mesh with inset showing anisotropic elements and definition of piecewise constant coefficients on subdomains.
      Right panel: illustration of solution with constant forcing term.}
   \label{fig:crooked-pipe}
\end{figure}

We use the same problem setup as in \cite{Pazner2023}, and solve this problem with increasing polynomial degree.
We compare the performance of the saddle-point solvers discussed presently to applying the ADS preconditioner directly to the high-order assembled matrix, and to ADS applied to the spectrally equivalent low-order-refined system.
We first compare CPU results computed using 144 MPI ranks on four nodes of LLNL's \textit{Quartz} supercomputer, which is the same configuration considered in \cite{Pazner2023}.
The results are shown in \Cref{tab:crooked-pipe-cpu}.
Both methods display a pre-asymptotic increase in the number of iterations with increasing polynomial degree; no increase in the number of iterations for the saddle-point solver was observed when increasing the polynomial degree from $p=5$ to $p=6$, suggesting that the asymptotic regime has been reached.
At $p=6$, the saddle-point solver required $1.9\times$ as many iterations as at $p=2$, and the LOR--ADS solver required $2.5\times$ as many iterations.
Overall, the saddle-point solver required approximately $1.5\times$ as many iterations as the LOR--ADS solver, however it was faster overall by a factor of 6--9$\times$.
The reason for this speed-up (despite a larger number of iterations) is that the saddle-point preconditioner is much less computationally expensive than the ADS preconditioner:
each application of the saddle-point preconditioner requires just one diagonal scaling and one scalar AMG V-cycle (applied to the smaller $L^2$ system), whereas each application of ADS requires (depending on the specific solver settings) around 11 AMG V-cycles.

\begin{table}
   \centering
   \caption{CPU results from the crooked pipe test problem (144 CPU cores on LLNL's \textit{Quartz}).}
   \label{tab:crooked-pipe-cpu}

   \small

   \begin{tabular}{c|cS|cS|cr}
      \toprule
      & \multicolumn{2}{c|}{Saddle-Point} & \multicolumn{2}{c|}{LOR--ADS} \\
      $p$ & Its. & {Time (s)} & Its. & {Time (s)} & Speed-up & \hdr{\# DOFs} \\
      \midrule
      2 & 168 & 0.26 & 96 & 1.73 & 6.65$\times$ & 356{,}500 \\
      3 & 231 & 0.87 & 143 & 5.73 & 6.59$\times$ & 1{,}190{,}115 \\
      4 & 269 & 2.11 & 172 & 14.91 & 7.07$\times$ & 2{,}805{,}520 \\
      5 & 299 & 4.45 & 206 & 36.90 & 8.29$\times$ & 5{,}461{,}375 \\
      6 & 323 & 8.73 & 241 & 80.84 & 9.26$\times$ & 9{,}416{,}340 \\
      \bottomrule
   \end{tabular}

   \vskip\floatsep
   \caption{GPU results from the crooked pipe test problem (4 Nvidia V100 GPUs on LLNL's \textit{Lassen}).}
   \label{tab:crooked-pipe-gpu}

   \begin{tabular}{c|cS|cS|cr}
      \toprule
      & \multicolumn{2}{c|}{Saddle-Point} & \multicolumn{2}{c|}{LOR--ADS} \\
      $p$ & Its. & {Time (s)} & Its. & {Time (s)} & Speed-up & \hdr{\# DOFs} \\
      \midrule
      2 & 193 & 0.85 & 143 & 7.73 & 9.09$\times$ & 356{,}500 \\
      3 & 251 & 1.52 & 206 & 14.97 & 9.85$\times$ & 1{,}190{,}115 \\
      4 & 298 & 2.19 & 269 & 23.72 & 10.83$\times$ & 2{,}805{,}520 \\
      5 & 335 & 3.49 & 332 & 37.60 & 10.77$\times$ & 5{,}461{,}375 \\
      6 & 360 & 4.01 & 392 & 61.94 & 15.45$\times$ & 9{,}416{,}340 \\
      \bottomrule
   \end{tabular}
\end{table}

These tests are repeated on the GPU using one node of LLNL's \textit{Lassen} with 4 Nvidia V100 GPUs.
The results are shown in \Cref{tab:crooked-pipe-gpu}.
The iteration counts for the GPU test are larger than those obtained on the CPU.
This is because different BoomerAMG settings are used when \textit{hypre} is run on the CPU or GPU \cite{Falgout2021}.
In particular, hybrid Gauss-Seidel smoothing is used on the CPU, while $\ell_1$-Jacobi smoothing is used on the GPU.
This appears to have a more significant detrimental effect on the ADS solver; the iteration counts for the saddle-point solver are largely similar to the CPU case.
The saddle-point solver consistently achieves 9--15$\times$ speed-up compared with the LOR--ADS solver, with greater speed-up at higher orders.
At $p=6$, the total solve time using the saddle-point solver using the GPU on one node of \textit{Lassen} is over twice as fast as on the CPU using 4 nodes of \textit{Quartz}.

\begin{table}
   \centering
   \caption{Matrix-based methods: CPU results from the crooked pipe test problem (144 CPU cores on LLNL's Quartz).}
   \label{tab:crooked-pipe-matrix}

   \small

   \begin{tabular}{c|cS|cS|r}
      \toprule
      & \multicolumn{2}{c|}{Hybridization} & \multicolumn{2}{c|}{ADS} \\
      $p$ & Its. & {Time (s)} & Its. & {Time (s)} & \hdr{\# DOFs} \\
      \midrule
      2 & 28 & 0.16 & 74 & 2.49 & 356{,}500 \\
      3 & 35 & 0.86 & 85 & 19.20 & 1{,}190{,}115 \\
      4 & 36 & 5.73 & 92 & 100.77 & 2{,}805{,}520 \\
      5 & 31 & 29.82 & 102 & 390.32 & 5{,}461{,}375 \\
      6 & 36 & 245.54 & 112 & 1293.12 & 9{,}416{,}340 \\
      \bottomrule
   \end{tabular}
\end{table}

For comparison, we also include CPU results for the matrix-based algebraic hybridization and ADS solvers in \Cref{tab:crooked-pipe-matrix};
GPU results are not reported because GPU-accelerated versions of these algorithms are not available.
Consistent with the results reported in \cite{Dobrev2019}, the algebraic hybridization solver greatly outperforms the ADS solver on all cases considered.
Similarly, the LOR--ADS solver outperforms the ADS solver for $p \geq 2$, as reported in \cite{Pazner2023}, and consequently the saddle-point solver also compares favorably to matrix-based ADS for these test cases.
The comparison with algebraic hybridization is more interesting.
The hybridization solver (which performs a conjugate gradient iteration directly on the hybridized Schur complement system) requires significantly fewer iterations than either ADS or the saddle-point solver.
For $p=2$, hybridization results in the fastest overall runtime, and for $p=3$, we obtained roughly equal runtimes with the hybridization and saddle-point solvers.
However, for $p > 3$, the time required to assemble and set up the hybridized solver becomes prohibitive, and the saddle-point solver results in significant speed-ups.
For $p=4$, the saddle-point solver is almost twice as fast, and for $p=6$, the saddle-point solver results in an almost $30\times$ speedup.
We conclude that for relatively low orders, when assembling and forming the hybridized Schur complement system is not prohibitively expensive, algebraic hybridization is a highly competitive option.
For higher orders, matrix-free methods such as the saddle-point solver discussed presently become necessary, and we note that the benefits of the matrix-free methods, such as reduced memory footprint, are more pronounced on GPU-based architectures.

\subsection{Darcy flow: SPE10 benchmark}

In this section, we consider the SPE10 benchmark problem from the reservoir simulation community \cite{Christie2001}.
This problem has been widely used as a test case for solvers in the literature, see e.g.\ \cite{Kanschat2017,Dobrev2019}.
The problem geometry is $\Omega = [0, 1200] \times [0, 2200] \times [0, 170]$, represented as a Cartesian mesh with $60 \times 220 \times 85$ elements.
The permeability tensor is defined as a piecewise constant on each element.
This coefficient is highly heterogeneous, exhibiting contrast of over $10^7$; see \Cref{fig:spe10}.
We solve Darcy's equations for the velocity and pressure,
\[
\begin{aligned}
   \left\{
   \begin{aligned}
      \kappa \bm u + \nabla p &= 0, \\
      \nabla \cdot \bm u &= 0, \\
   \end{aligned}
   \right. &\qquad\text{in $\Omega$}. \\
   \bm u \cdot \bm n = (1, 0, 0)^\tr \cdot \bm n &\qquad\text{on $\partial\Omega$,}
\end{aligned}
\]
where $\kappa$ is inverse permeability matrix.
Since pure Neumann boundary conditions are enforced for this problem, the Schur complement is singular, with a nullspace consisting of all constant functions.
To ensure a convergent solver, every application of the approximate inverse $\widehat{\S}^{-1}$ is followed by an orthogonalization step.

\begin{figure}
   \centering
   \includegraphics[width=4cm]{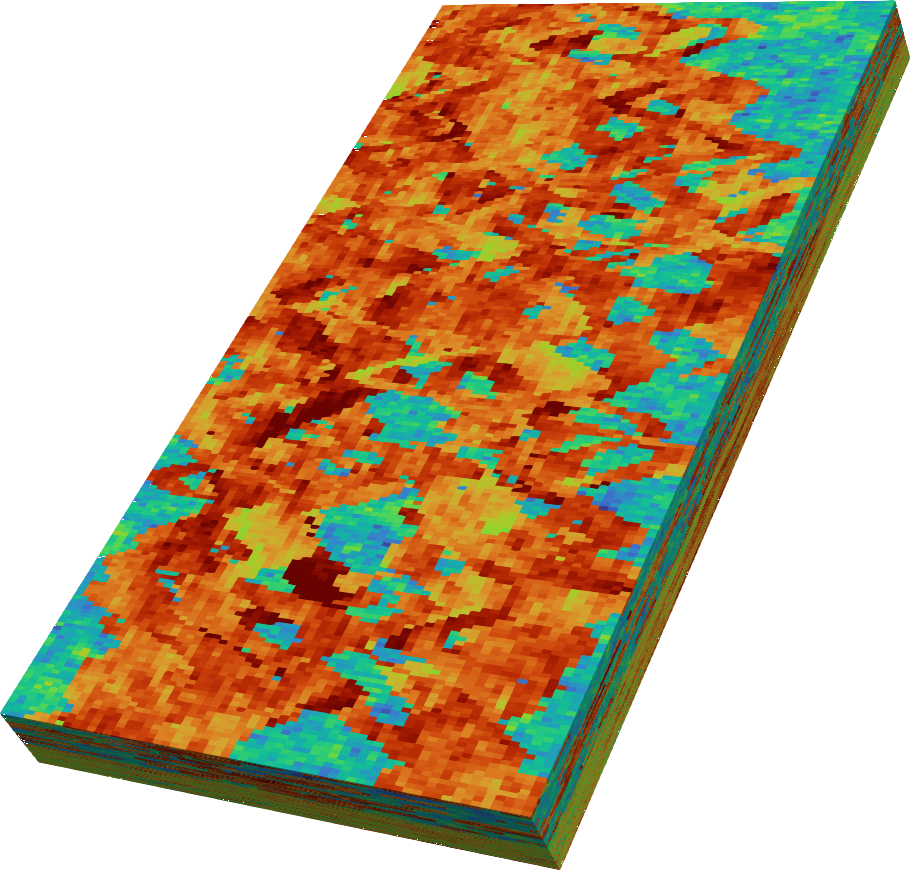}\hspace{12pt}
   \raisebox{0.4cm}{\includegraphics[scale=0.7]{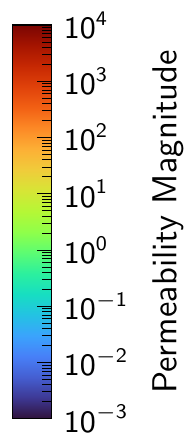}}
   \hspace{3em}
   \includegraphics[width=4cm]{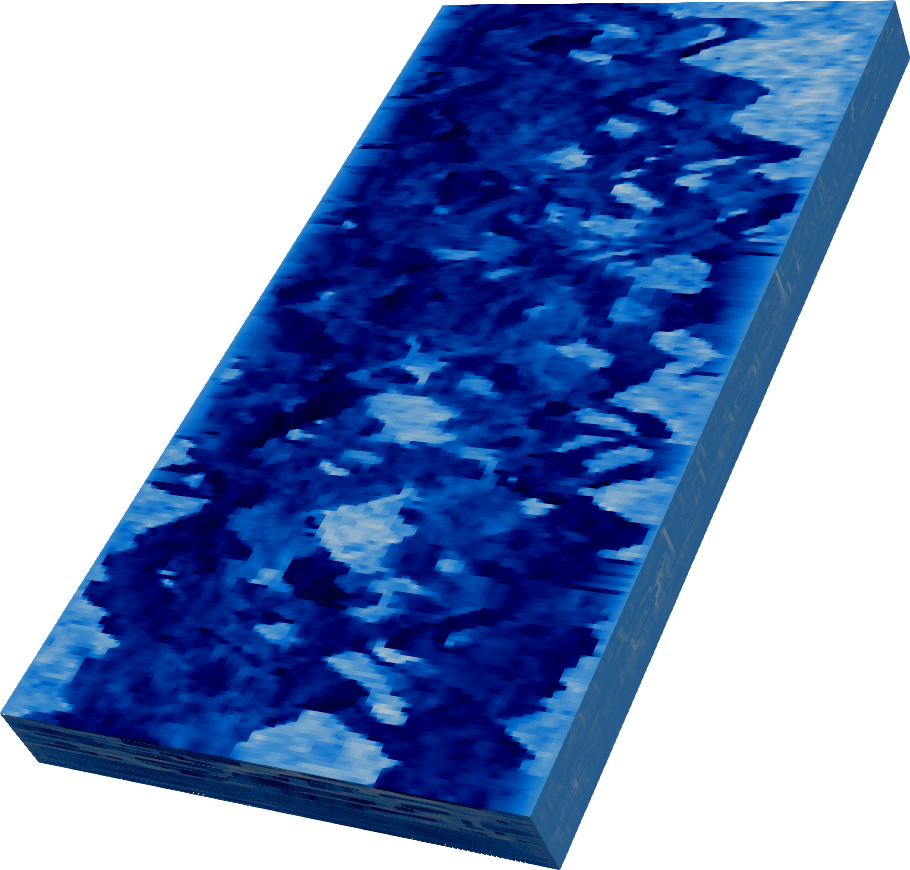}\hspace{12pt}
   \raisebox{0.4cm}{\includegraphics[scale=0.7]{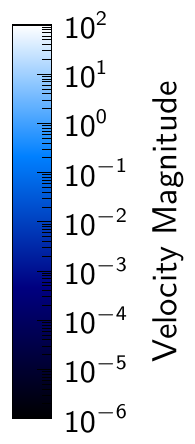}}
   \caption{SPE10 problem, logarithmic color scale (cf.\ \cite{Christie2001}). Left: magnitude of permeability coefficient. Right: solution velocity magnitude.}
   \label{fig:spe10}
\end{figure}

We use this problem to perform a weak scaling study using $p=4$ and $p=6$ elements.
With each refinement, the number of degrees of freedom is roughly doubled by solving on a finer mesh, and the problem is solved using twice the number of GPUs.
With increasing mesh refinement the iteration counts remain roughly constant.
The largest problem has roughly $7 \times 10^8$ unknowns, and is solved in about 10 seconds on 256 GPUs.
For each of the cases considered, the solve time and iterations are similar between the $p=4$ and $p=6$ cases, illustrating the uniform convergence of the solver with respect to polynomial degree.
The parallel scalability of the solver is largely consistent with similar results reported in \cite{Kolev2021a,LORGPU2022}.
For the largest problem, the weak parallel efficiency of the solver is 31\%, comparing favorably to the parallel efficiency of 25\% for the algebraic hybridization solver and 7.5\% for the ADS solver reported for similarly sized SPE10 problems in \cite{Dobrev2019}.

\begin{table}
   \centering
   \caption{Scalability study for the SPE10 benchmark problem.}
   \label{tab:spe10}
   \small
   \begin{tabular}{cccc}
      \toprule
      \multicolumn{4}{c}{$p=4$}\\
      \midrule
      \# GPUs & \# DOFs & Its. & Time (s) \\
      \midrule
      1 & $3.36\times10^{6}$ & 147 & 3.25 \\
      2 & $5.76\times10^{6}$ & 155 & 4.31 \\
      4 & $1.20\times10^{7}$ & 157 & 5.15 \\
      8 & $2.68\times10^{7}$ & 160 & 6.45 \\
      16 & $4.58\times10^{7}$ & 159 & 6.35 \\
      32 & $9.53\times10^{7}$ & 164 & 6.94 \\
      64 & $2.13\times10^{8}$ & 164 & 8.62 \\
      128 & $3.66\times10^{8}$ & 166 & 8.55 \\
      256 & $7.61\times10^{8}$ & 165 & 10.44 \\
      \bottomrule
   \end{tabular}%
   \hspace{3em}
   \begin{tabular}{cccc}
      \toprule
      \multicolumn{4}{c}{$p=6$}\\
      \midrule
      \# GPUs & \# DOFs & Its. & Time (s) \\
      \midrule
      1 & $3.30\times10^{6}$ & 165 & 3.14 \\
      2 & $5.87\times10^{6}$ & 166 & 4.36 \\
      4 & $1.13\times10^{7}$ & 166 & 4.82 \\
      8 & $2.63\times10^{7}$ & 165 & 6.41 \\
      16 & $4.67\times10^{7}$ & 165 & 6.19 \\
      32 & $9.01\times10^{7}$ & 171 & 7.20 \\
      64 & $2.10\times10^{8}$ & 176 & 8.97 \\
      128 & $3.73\times10^{8}$ & 173 & 8.79 \\
      256 & $7.20\times10^{8}$ & 175 & 10.53 \\
      \bottomrule
   \end{tabular}
\end{table}

\subsection{Nonlinear radiation diffusion}

As a final test case, we consider the solution of time-dependent grey (one-group) nonlinear radiation diffusion equations.
The governing equations are
\begin{align}
   \label{eq:rad-diff-e}
   \rho \frac{\partial e}{\partial t} &= -c \sigma (a T_{\mathrm{mat}}^4 - E) + Q, \\
   \label{eq:rad-diff-E}
   \frac{\partial E}{\partial t} + \nabla \cdot F &= c\sigma (a T_{\mathrm{mat}}^4 - E) + S, \\
   \label{eq:rad-diff-F}
   F &= -\frac{c}{3 \sigma} \nabla E,
\end{align}
where $e$ is the material specific internal energy, $E$ is the radiation energy density, and $F$ is the radiation flux, $\sigma$ is the absorption opacity, $c$ is the speed of light, $a$ is the black body constant, and $\rho$ is the material density.
$Q$ and $S$ are given source terms.
The material temperature $T_{\mathrm{mat}}$ is related to the internal energy $e$ through the specific heat $C_v$.
The energy unknowns $e$ and $E$ are taken in the $L^2$ finite element space $W_h$, and the radiation flux is taken in the $\Hdiv$ space $\Vh$.

For this test case, we use the manufactured solution proposed in \cite{Brunner2006}.
The equations are integrated in time implicitly; we use the three-stage $L$-stable singly diagonal implicit Runge--Kutta (SDIRK) method \cite{Alexander1977}.
The resulting nonlinear algebraic equations are solved using an iterative scheme as follows.
First, the radiation flux $F$ is lagged, and a coupled nonlinear system for the material energy $e$ and radiation energy $E$ is solved.
Then, the material energy $e$ is fixed, and a Darcy-type linear system is solved for the radiation energy $E$ and radiation flux $F$.
This procedure is repeated until the full nonlinear residual has converged to the desired tolerance;
for the time-accurate time steps considered presently, this method typically requires only one or two outer iterations to reach a relative tolerance of $10^{-12}$.

For fixed radiation flux $F$, the (inner) nonlinear equations corresponding to \eqref{eq:rad-diff-e} and \eqref{eq:rad-diff-E} are solved element-by-element using Newton's method.
The resulting Jacobian matrix takes the form of a $2\times2$ element-wise block diagonal system, which is well-preconditioned by a block Jacobi preconditioner with blocks of size $2\times2$.
After solving for the energies $e$ and $E$, the material energy $e$ is fixed, and the coupled linear equations for $E$ and $F$ corresponding to \eqref{eq:rad-diff-E} and \eqref{eq:rad-diff-F} are solved;
the solution of this system is often a bottleneck in radiation hydrodynamics simulations, emphasizing the importance of the development of efficient preconditioners \cite{Gonzalez2015,Langer2015}.
These equations take the form of a Darcy-type problem, which can be preconditioned using the block preconditioners described in this work.

The test problem in \cite{Brunner2006} is posed in spherical coordinates; presently, we solve this problem in Cartesian coordinates on a three-dimensional mesh with 884{,}736 hexahedral elements using polynomial degree $p=4$.
This problem configuration has $1.7\times10^8$ Raviart--Thomas degrees of freedom and $5.7\times10^7$ $L^2$ degrees of freedom, and is solved in parallel with 16 Nvidia V100 GPUs on four nodes of LLNL's \textit{Lassen}.
The equations are integrated until a final time of $0.1 / \tau \approx 4.39 \times 10^{-11}$, where $\tau \approx 2.28 \times 10^9$ is the decay constant of the test problem.
The time step is chosen to be $\Delta t = h / \tau$, and 10 implicit time steps (each with three Runge--Kutta stages) are taken.
Each stage requires an outer nonlinear solve (with relative tolerance $10^{-12}$), an inner nonlinear solve (with relative tolerance $10^{-8}$), and a linear solver (with relative tolerance $10^{-8}$).
At the final time, the relative $L^2$ error was $7.72\times10^{-8}$ in the material energy and $5.25\times10^{-7}$ in the radiation energy.
Each Runge--Kutta stage required two outer solver iterations and four inner Newton iterations.
In total, 60 saddle-point linear solves were performed, and on average the linear solver converged in 84 MINRES iterations.
The timing results and iteration counts are presented in the right panel of \Cref{fig:rad-diff}.
Consistent with results reported in \cite{Gonzalez2015}, approximately 92\% of the total runtime for this problem was spent in the linear solver.


\begin{figure}
   \centering
   \begin{minipage}{0.45\linewidth}
      \centering
      \raisebox{12pt}{\includegraphics[width=3cm]{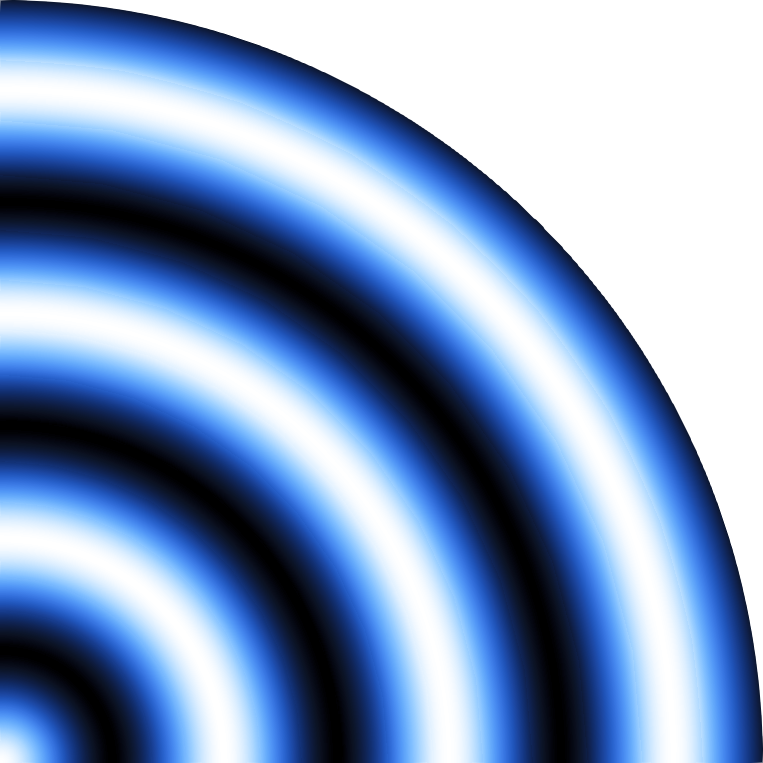}}
      \hspace{0.25cm}
      \includegraphics[scale=0.7]{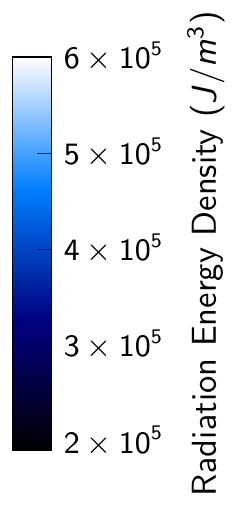}
   \end{minipage}%
   \begin{minipage}{0.5\linewidth}
      \centering
      \begin{tabular}{lSS}
         \toprule
         Alg.~Component & {Time (s)} & {\%} \\
         \midrule
         Nonlinear residual & 8.22 & 4.45\% \\
         Inner Newton solver & 7.15 & 3.87\% \\
         \multicolumn{3}{l}{\quad\small (2 iterations per solve)} \\
         Linear solver & 169.39 & 91.68\% \\
         \multicolumn{3}{l}{\quad\small (84 iterations per solve)} \\
         \bottomrule
      \end{tabular}
   \end{minipage}
   \caption{
      Left: snapshot of radiation energy density from radiation diffusion test case.
      Right: timing results and iteration counts on 16 V100 GPUs.
   }
   \label{fig:rad-diff}
\end{figure}

\section{Conclusions}
\label{sec:conclusions}

In this work, we have developed solvers for grad-div and Darcy-type problems in $\Hdiv$ based on block preconditioning for the associated saddle-point system.
Properties of the interpolation--histopolation high-order basis result in off-diagonal blocks with enhanced sparsity.
The Schur complement of the block system can be approximated by an M-matrix with sparsity independent of the polynomial degree;
such systems are amenable to preconditioning with scalable algebraic multigrid methods.
The $(1,1)$-block can be well-approximated by a diagonal matrix.
The resulting block preconditioners result in uniform convergence, independent of mesh size and polynomial degree.
These solvers are amenable to GPU acceleration; high-throughput algorithms are developed and benchmarked.
The performance of the solvers is demonstrated on several benchmark problems, including the ``crooked-pipe'' grad-div problem, the SPE10 benchmark from reservoir simulation, and nonlinear radiation-diffusion.

\section{Acknowledgements}

This work was performed under the auspices of the U.S.\ Department of Energy by Lawrence Livermore National Laboratory under Contract DE-AC52-07NA27344 (LLNL-JRNL-848035).
W.\ Pazner and Tz.\ Kolev were partially supported by the LLNL-LDRD Program under Project No.~20-ERD-002.

\ifsiam
   \small
   \bibliographystyle{siamplain}
   \bibliography{refs}
\else
   \printbibliography
\fi

\end{document}